\documentclass[10pt]{article}

\usepackage[utf8]{inputenc}
\usepackage[T1]{fontenc}

\usepackage{epsf}
\usepackage{amsmath}

\allowdisplaybreaks

\usepackage[showframe=false]{geometry}
\usepackage{changepage}

\usepackage{epsfig}
\usepackage{amssymb}

\usepackage{amsthm}
\usepackage{setspace}
\usepackage{cite}
\usepackage{mcite}

\usepackage{algorithmic}  
\usepackage{algorithm}

\usepackage{shadow}
\usepackage{fancybox}
\usepackage{fancyhdr}

\usepackage{color}
\usepackage[usenames,dvipsnames,svgnames,table]{xcolor}

\definecolor{mypurple}{rgb}{.4,.0,.5}

\usepackage[hyphens]{url}

\usepackage[colorlinks=true,
            linkcolor=black,
            urlcolor=blue,
            citecolor=purple]{hyperref}

\usepackage{breakurl}

\def\s{{\bf s}}

\def\y{{\bf y}}
\def\v{{\bf v}}
\def\x{{\bf x}}

\def\x{{\mathbf x}}

\def\s{{\bf s}}

\def\v{{\bf v}}
\def\x{{\bf x}}
\def\y{{\bf y}}
\def\z{{\bf z}}

\def\be{\begin{equation}}
\def\ee{\end{equation}}
\def\ba{\left[\begin{array}}
\def\ea{\end{array}\right]}

\def\s{{\bf s}}

\def\v{{\bf v}}
\def\x{{\bf x}}
\def\y{{\bf y}}
\def\z{{\bf z}}

\def\1{{\bf 1}}

\def\g{{\bf g}}
\def\0{{\bf 0}}

\def\cG{{\mathcal G}}

\def\mR{{\mathbb R}}

\def\mN{{\mathbb N}}
\def\mE{{\mathbb E}}
\def\mS{{\mathbb S}}

\def\lp{\left (}
\def\rp{\right )}

\def\s{{\bf s}}

\def\y{{\bf y}}
\def\v{{\bf v}}
\def\x{{\bf x}}

\def\x{{\mathbf x}}

\def\s{{\bf s}}

\def\v{{\bf v}}
\def\x{{\bf x}}
\def\y{{\bf y}}
\def\z{{\bf z}}

\def\be{\begin{equation}}
\def\ee{\end{equation}}
\def\ba{\left[\begin{array}}
\def\ea{\end{array}\right]}

\def\s{{\bf s}}

\def\v{{\bf v}}
\def\x{{\bf x}}
\def\y{{\bf y}}
\def\z{{\bf z}}

\def\({\left (}
\def\){\right )}

\def\1{{\bf 1}}

\def\g{{\bf g}}
\def\0{{\bf 0}}

\def\cX{{\mathcal X}}

\def\cL{{\mathcal L}}

\def\cN{{\mathcal N}}
\def\cX{{\mathcal X}}

\usepackage{xcolor}
\usepackage{color}

\definecolor{darkgreen}{rgb}{0, 0.4,0}

\definecolor{purplebrown}{rgb}{0.5,0.1,0.6}

\definecolor{ultclupcol}{rgb}{0.1,0.5,0.5}

\definecolor{mytrycolor}{rgb}{0.5,0.7,0.2}

\definecolor{ultclupcola}{rgb}{.5,0,.5}

\definecolor{shadebrown}{rgb}{0.1,0.1,0.9}
\definecolor{lightblue}{rgb}{0.2,0,1}

\usepackage{fancybox}
\usepackage{graphicx}
\usepackage{epstopdf}
\usepackage{epsfig}
\usepackage{wrapfig}
\usepackage{subfigure}

\usepackage{xcolor}
\usepackage{tcolorbox}

\newtcbox{\xmybox}{on line,
arc=7pt,
before upper={\rule[-3pt]{0pt}{10pt}},boxrule=0pt,
boxsep=0pt,left=6pt,right=6pt,top=0pt,bottom=0pt,enhanced, coltext=blue, colback=white!10!yellow}

\newtcbox{\xmyboxa}{on line,
arc=7pt,
before upper={\rule[-3pt]{0pt}{10pt}},boxrule=0pt,
boxsep=0pt,left=6pt,right=6pt,top=0pt,bottom=0pt,enhanced, colback=white!10!yellow}

\newtcbox{\xmyboxb}{on line,
arc=7pt,
before upper={\rule[-3pt]{0pt}{10pt}},boxrule=1pt,colframe=darkgreen!100!blue,
boxsep=0pt,left=6pt,right=6pt,top=0pt,bottom=0pt,enhanced, colback=white!10!yellow}

\newtcbox{\xmyboxc}{on line,
arc=7pt,
before upper={\rule[-3pt]{0pt}{10pt}},boxrule=.7pt,colframe=blue!100!blue,
boxsep=0pt,left=6pt,right=6pt,top=0pt,bottom=0pt,enhanced, coltext=blue, colback=white!10!yellow}

\newtcbox{\xmytboxa}{on line,
arc=7pt,
before upper={\rule[-3pt]{0pt}{10pt}},boxrule=.0pt,colframe=pink!50!yellow,
boxsep=0pt,left=6pt,right=6pt,top=0pt,bottom=0pt,enhanced, coltext=white, colback=blue!40!red}

\newtcbox{\xmytboxb}{on line,
arc=7pt,
before upper={\rule[-3pt]{0pt}{10pt}},boxrule=.0pt,colframe=pink!50!yellow,
boxsep=0pt,left=6pt,right=6pt,top=0pt,bottom=0pt,enhanced, coltext=white, colback=white!40!green}

\makeatletter
\newcommand\subsubsubsection{\@startsection{paragraph}{4}{\z@}{-2.5ex\@plus -1ex \@minus -.25ex}{1.25ex \@plus .25ex}{\normalfont\normalsize\bfseries}}
\newcommand\subsubsubsubsection{\@startsection{subparagraph}{5}{\z@}{-2.5ex\@plus -1ex \@minus -.25ex}{1.25ex \@plus .25ex}{\normalfont\normalsize\bfseries}}
\makeatother

\newtheorem{theorem}{Theorem}

\newtheorem{remark}{Remark}

\begin{document}

\begin{singlespace}

\title {Precise sample covariance spectral norm error -- an RDT view   
}
\author{
\textsc{Mihailo Stojnic
\footnote{e-mail: {\tt flatoyer@gmail.com}}
}}
\date{}
\maketitle

\centerline{{\bf Abstract}} \vspace*{0.1in}

We study the sample covariance error of centered Gaussians. A remarkable breakthrough \cite{KolLou17}  established the correct error scaling order and explicitly revealed the critical role of both the effective rank and the true covariance spectrum.

In this work, we move beyond scaling characterizations and determine the precise limiting value of the error's spectral norm. To do so, we develop a generic framework based on Random Duality Theory (RDT). Within this framework, we first determine closed-form, explicit RDT-based upper bounds. We then establish complementary lower bounds by introducing a novel bilinear-quadratic RDT lower-bounding mechanism. By combining this mechanism with a two-replica systems bounding strategy, we show that our lower and upper bounds match in large-dimensional contexts. Our theoretical results are supplemented with numerical evaluations and simulations, demonstrating an excellent agreement already for problem sizes on the order of thousands.

\vspace*{0.25in} \noindent {\bf Index Terms: Sample covariance error; Random duality theory (RDT)}.

\end{singlespace}

\section{Introduction}
\label{sec:back}

Covariance estimation is a classical task of critical importance across a variety of scientific and engineering fields. Applications range from machine learning, graphical models, and compressed sensing \cite{WainwrightBook19} to wireless communications, signal processing \cite{KrimVib96,HaghCaire2017,CaireCov20,Stoica99,Raskutti10}, and image analysis \cite{DahKeyPitz00,ZhangSch10,ChenRWK20}. Furthermore, it is widely utilized in finance \cite{LedoitWolf03,LedoitWolf04,Engle02,Holtz11,BaiShi11,FanRigWang15}, bioinformatics, genomics, microarrays \cite{XieBentler03,SchafStr05,HeroRaj12,Friedman08,LangHor08,LedoitWolf04,Witten09}, and generic random matrix theory \cite{BaiSilver10}.

The most well-known estimator is the classical \emph{sample covariance}, $\hat{\Sigma}$, which has been extensively studied over the last several decades. Typically, the primary interest lies in its deviation from the true covariance, $\Sigma$, and how this deviation is impacted by sample complexity \cite{KanLovSim97,Bourgain99,Rudelson99,GianHarTso05,Paouris06,Adam10,Adamczak11,KolLou17,VershCov12}. Because these underlying problems are challenging, analytical error estimates are usually qualitative bounds focused on establishing correct scaling orders with respect to the ambient dimension, $d$, and sample size, $n$.

Here, we introduce a different methodology that allows for precise characterizations beyond scaling orders, enabling explicit descriptions of how sample complexity impacts estimation error. For example, in practically dominant large-dimensional contexts with a sample complexity ratio $\alpha=\frac{n}{d}$, our results can determine the exact value of the sample covariance error's spectral norm, $|\hat{\Sigma}-\Sigma|_2$. This allows us to answer highly practical questions—such as how much the error decreases if the sample size is doubled or tripled—which are typically difficult to resolve within traditional scaling-order frameworks.

\subsection{Relevant prior work}
\label{sec:examples}


\noindent \underline{\textbf{\emph{Classical scenarios:}}}  Various metrics of sample covariance error have been of interest. Early studies, dating back to the 1990s, focused on quantifying the error's operator norm in the context of volume estimation and convex bodies \cite{KanLovSim97,Bourgain99,Rudelson99}. The initial upper-bounding estimate $O(\frac{d}{\sqrt{n}})$ \cite{KanLovSim97} related to isotropic vectors was successively improved in \cite{KanLovSim97, Bourgain99, Rudelson99, GianHarTso05, Paouris06, Adam10} and eventually brought to the expected $O(\sqrt{\frac{d}{n}})$ in \cite{Adamczak11} for a generic class of ensembles, including log-concave ones. Similar results of $O(\sqrt{\frac{d}{n}}, \frac{d}{n})$ were obtained in \cite{VershNonAsym12} for a generic non-isotropic scenario where the true covariance matrix is different from identity (throughout this presentation, we are typically interested in sufficiently over-sampled scenarios, which are likely the most practically relevant; in such contexts, $O(\sqrt{\frac{d}{n}}, \frac{d}{n}) = O(\sqrt{\frac{d}{n}})$). This is scaling-wise optimal in typical scenarios where data are highly heterogeneous. However, when facing more homogeneous data where the covariances reside in a small dimensional subspace (i.e., they are of a fairly small rank), one would expect that the $d$-dependence might be improved.

\vspace{.08in} 

\noindent \underline{\textbf{\emph{Rank deficient $\Sigma$:}}}  A strong effort has been made to not only improve, but completely remove the $d$-dependence. Relying on the so-called \emph{effective rank} as a (dimensionless) measure of rank deficiency, \cite{Tropp11,VershNonAsym12} observed that in the above results $d$ can be decreased to the product of effective rank and $\log(d)$ if the support of the underlying distribution is limited to the Euclidean ball. Analogous results were obtained through the consideration of covariance estimation with missing observations \cite{LouniciMiss14}. Nonetheless, while weaker, the $d$-dependence still remained.

The complete, formal removal of $d$-dependence was finally achieved in a subgaussian context in the breakthrough work \cite{KolLou17}, where a generic chaining methodology \cite{MendelsonChain10,KlarMendel05,TalChainBook05} (see also, \cite{bednorz16,bednorz14,dirkchain2015,Mendeldiscchain11}) was utilized to obtain an error characterization dependent solely on the effective rank and $n$ (see also, \cite{Adamczak15} for a reproof). Subsequently, the results obtained in \cite{KolLou17} have been significantly extended so that their correspondingly modified forms cover smooth functions of the error's operator norm, spectral covariance projectors, and shift models \cite{koltch18,Koltch21,Koltch22,KoltchZhil21,KoltchLouProj17} (for extensions in different technical directions, see \cite{Zhiv24,qihan24}).

The results of \cite{PuchNosSpo25,BuneaXiao15} consider the so-called Frobenius norm as another type of error metric and prove effective rank bounding estimates analogous to the operator norm ones from \cite{KolLou17}. For the importance and relevance of the Frobenius norm error metric in hypothesis testing, see, e.g., \cite{LedoitWolf02,HanWu20}.

\vspace{.07in}

\noindent \underline{\textbf{\emph{Various (distributional, parametric, structural, contaminated, etc.) extensions:}}}  Many of the results discussed can be adapted to hold in similar or even improved forms when $\Sigma$ is a priori known to possess a particular structure. These adaptations include both potential estimator restructuring and refinements in theoretical analyses. A variety of structures have been considered throughout recent literature, including: sparse \cite{BickelLev08, CaiZhou12, FanRigWang15, Karoui08}, block \cite{Perrot22, MarkBlock24}, Toeplitz \cite{XiaoWu12, Caitoeplitz13}, and Kronecker \cite{TsilHero13, LengPan18}. For more on general covariance structuring and its relevance, please see \cite{CaiSurvey16}.

Different distributional aspects, including anisotropic and heavy-tails, have also been of interest \cite{MinasyanZhiv23, OliRico24, Tikhom18, abdzhiv26, SriVersh13, Youssef13, MinskerWei20, MendZhiv20}. Additionally, deviations from distributional identicalness—with a particular emphasis on independent scenarios with different variance profiles—have been addressed in \cite{BandHand16, VanHandelSpec17, LatHandYouss18, CaiHanZhang22}. In this regard, particularly valuable progress has been achieved in studying matrix concentrations via free probability \cite{bandfree23, bandfree24} (for universality aspects, please refer to \cite{BrailUniv24}).

One often encounters scenarios where the available data is gathered in a non-ideal fashion. In these situations, the robustness properties of estimators become of prevalent interest. Recent literature has addressed robust covariance estimation in many non-ideal scenarios. This includes typical cases where the data are incomplete or missing \cite{CaiZhangMinmmax16,CaiZhang24,LouniciMiss14,Agost22,LouniciPacMiss23,AbdallaMiss24}, or corrupted and contaminated in various ways \cite{KeMRSZ19,ChenGaoRen18,MinskerWang24,Giulini18,abdzhiv26,SriVersh13,Youssef13}. Corresponding robust mean estimations are equally relevant. Many of the techniques developed for handling either covariance or mean estimation can be reutilized to address the other. For recent progress, please see, for example, \cite{DiakKanePen20,MinasyanZhiv23,Huber64,LugMend21,LaiRaoVempala16,DalMin2022,LugMendSurvey19}. More information on classical corruptive and contaminated models can be found in standard robust estimation references, such as \cite{Huber81,Maronna19,Hampel86,Huber64}.

There is a host of problems where data is structured differently and similar, but often more complicated, estimation tasks are faced. These include tensor or cross-covariance considerations \cite{alghattas25,hanwillzhang22,ZhouZhu21,DiakKane25,ChenSanz26}, multi-inference alignment \cite{DouFanZhou24,PerryNBRS19}, and (sparse) principal components/SVD analysis \cite{KoltchLofNic20,KoltchLouPca17,ZhangHan19}. In each of these scenarios, one can establish analogous types of estimators and characterize the residual estimation error in a fashion similar to traditional covariance estimation.

The problems we are studying here are to some degree related to topics in statistical inference and spectral analysis as well. Characterizing the properties of spiked models, recovering low-rank matrices/tensors from noisy observations, and analyzing the spectra of deformed random matrices are prominent examples where excellent progress has been made over the last decade. Determining their key features, residual estimation errors, BBP phase transitions \cite{BBP05,Peche06,BirGuioSpike20,MontRich14,DonGavJohn18,Lesetal17,PerryWB20,BarbMM17,PourBM24,GuiKKZ25,BehneReeves22,PakKK23,bandfree24}, and the locations of spectral edges—along with associated large deviations principles (LDP) (\cite{Guionnet2012large,Maida07,MckennaDeformed21,HussMcKenna24,MaillardLDP21,MergnyPott22})—typically requires a highly nontrivial transition from scaling to precise analyses.

\subsection{Our contributions}
\label{sec:cont}

 As stated earlier, we pursue a different direction here, move beyond scaling estimates, and focus on the precise characterization of the error's spectral norm. In particular, we consider a proportional large-dimensional setup with $\alpha= \frac{n}{d}$ remaining fixed as $d$ and $n$ grow. For centered Gaussian vectors with true covariance $\Sigma$, we consider the sample covariance $\hat{\Sigma}$ and determine the limiting average spectral norm of the error $\delta(\alpha) =\lim_{d\rightarrow\infty} \|\hat{\Sigma} - \Sigma\|_2$.

We achieve this by introducing a generic framework based on Random Duality Theory (RDT). First, a closed-form, explicit RDT-based upper bound is established (Sections \ref{sec:randpr}--\ref{sec:handlerd}). We then create a novel \emph{bilinear-quadratic} RDT lower-bounding mechanism (Section \ref{sec:strrd}). By combining this mechanism with a 2-replica systems bounding strategy, we ultimately show that it matches the upper bounds (Section \ref{sec:strhandbc}). Numerical evaluations and simulations are also conducted, showing excellent agreement with the theoretical predictions even for relatively small problem sizes on the order of thousands (Figures \ref{fig:fig1}--\ref{fig:fig3} and Section \ref{sec:numeval}).

Almost all of the results discussed in the previous section are of the scaling-order type. To move beyond these scaling barriers, a fundamentally different analytical mechanism was needed. A key benefit of this newly developed machinery is its versatility; it can be used to handle many, if not all, of the scenarios studied in prior literature and discussed in Section \ref{sec:examples}. Since this is an introductory paper, we chose the most classical variant of the problem. Extensions to encompass more advanced scenarios rely on the same concepts presented here, but are technically problem-specific and will be discussed elsewhere.

\section{Sample covariance error -- mathematical preliminaries}
\label{sec:sampcov}

We start by introducing precise definitions of mathematical objects used throughout the presentation. Let $n\in\mN$ and  $d\in\mN$ be two positive integers such that
\begin{equation}\label{eq:amat1a0}
\alpha \triangleq \lim_{d\rightarrow\infty} \frac{n}{d}= const.
\end{equation}
We follow the traditional literature convention and refer to the high-dimensional regime characterized by (\ref{eq:amat1a0}) as linear or proportional. Let $\bar{\x}\in\mR^{d\times 1}$ be a $d$-dimensional vector comprised of centered Gaussian elements. Moreover, let the associated covariance be $\Sigma$, i.e., let \begin{equation}\label{eq:amat1a0a0}
\mE \bar{\x}\bar{\x}^T =\Sigma.
\end{equation}
We then also write $\bar{\x}\sim \cN(0,\Sigma)$. A classical way to estimate $\Sigma$ is to collect a sample $\{ \bar{\x}^{(1)},\bar{\x}^{(2)},\dots,\bar{\x}^{(n)}\}$ of $n$ independent draws of $\bar{\x}$ and to average them out. Formally, for $\bar{\x}^{(i)}\sim \cN(0,\Sigma),i\in\{1,2,\dots,n\}$, one sets \begin{equation}\label{eq:amat1a0a1}
X \triangleq \begin{bmatrix}
               \lp \bar{\x}^{(1)} \rp^T  \\ \lp \bar{\x}^{(2)} \rp^T \\ \vdots \\ \lp \bar{\x}^{(n)} \rp^T 
             \end{bmatrix},
\end{equation}
and considers the sample covariance matrix $\hat{\Sigma} \triangleq \frac{1}{n}X^TX$ as an approximation to $\Sigma$. In other words, one hopes that
\begin{equation}\label{eq:amat1a0a2}
\hat{\Sigma} \triangleq  \frac{1}{n} X^TX \approx \Sigma  \quad \mbox{ or } \quad   \|\hat{\Sigma} - \Sigma\| = \left \| \frac{1}{n}X^TX - \Sigma \right \| \mbox{ is small }, 
\end{equation}
where $\|\cdot\|$ is the norm of choice. Naturally, as the sample size, $n$, grows (which also means that $\alpha$ increases), the quality of the approximation is expected to improve. Precisely characterizing this improvement remains one of the fundamental challenges in statistical estimation and probability theory.

A major breakthrough in this area was achieved in \cite{KolLou17}, where the error's operator norm was shown to behave in the following way
\begin{equation}\label{eq:amat1a0a3}
 \delta_n \triangleq \mE \|\hat{\Sigma} - \Sigma\| \sim \|\Sigma\| \max \lp \sqrt{\frac{r(\Sigma)}{n}} , \frac{r(\Sigma)}{n}  \rp ,
\end{equation}
where
\begin{equation}\label{eq:amat1a0a4}
r(\Sigma) \triangleq \frac{\lp\mE \|\bar{\x}\|\rp^2}{\|\Sigma\|}.
\end{equation}
Accompanying concentration inequalities are established in \cite{KolLou17} as well. The estimation error introduced in (\ref{eq:amat1a0a3}) is a function of $n$, and larger sample sizes expectedly produce smaller errors (i.e., as $n$ increases, $\delta$ decreases). The characterizations in (\ref{eq:amat1a0a3}) and (\ref{eq:amat1a0a4}) provide the correct error scaling, highlighting the roles of the sample size, $n$, and the norm of the true covariance, $\|\Sigma\|$. While this provides foundational information regarding the quality of $\hat{\Sigma}$, practical scenarios often demand more. Specifically, when gathering abundant data is costly or difficult, precisely understanding the tradeoff between sample size and accuracy becomes a necessity.

In the remainder of the paper, we address this demand and determine the underlying tradeoff. In particular, specializing for concreteness to the typical spectral norm choice, we precisely characterize $\delta(\alpha)$, defined as 
\begin{equation}\label{eq:amat1a0a5}
\delta(\alpha)\triangleq  \lim_{n\rightarrow \infty }\delta_n =  \lim_{n\rightarrow \infty }  \mE \|\hat{\Sigma} - \Sigma\|_2 =  \lim_{d\rightarrow \infty }  \mE \|\hat{\Sigma} - \Sigma\|_2, 
\end{equation}
where the last equality holds based on (\ref{eq:amat1a0}).

Let the output of function $\lambda_i(\cdot)$ be the $i$-th smallest eigenvalue of its argument. Then one has
\begin{equation}\label{eq:amat1a0a6}
 \|\hat{\Sigma} - \Sigma\|_2 = \max(\lambda_n(\hat{\Sigma} - \Sigma),|\lambda_1(\hat{\Sigma} - \Sigma)|). 
\end{equation}
For the time being, we focus on the largest eigenvalue of $\lp \hat{\Sigma} -\Sigma   \rp$, i.e., on $\lambda_n\lp \hat{\Sigma} -\Sigma   \rp$ (later on, in Section \ref{sec:matchdellam}, we will see that such an approach suffices to determine $\delta(\alpha)$). To that end, we denote the unit sphere in $\mR^d$ by $\mS^d=\{\x|\|\x\|_2=1\}$ and write
\begin{eqnarray}\label{eq:inteq1ad0}
\lambda_n\lp \hat{\Sigma} -\Sigma   \rp 
= \max_{\x\in\mS^d} \hspace{.04in} \x^T \lp \hat{\Sigma} -\Sigma   \rp \x
= \max_{\x\in\mS^d}  \hspace{.04in}  \x^T \lp \frac{1}{n}X^TX -\Sigma   \rp \x
= \max_{\x\in\mS^d} \lp  \frac{1}{n}\x^T X^TX \x - \x^T\Sigma  \x \rp.
\end{eqnarray}
By definition, $\Sigma$ is positive semi-definite ($\Sigma\succeq 0$) and admits the following eigen-decomposition
\begin{eqnarray}\label{eq:inteq1ad1}
 \Sigma = USS^TU^T.
\end{eqnarray}
In (\ref{eq:inteq1ad1}),  $S\in\mR^{d\times d}$ is a $d\times d$ diagonal matrix (with roots of $\Sigma$'s eigenvalues on the main diagonal) and $U\in\mR^{d\times d}$ is a $d\times d$ orthogonal matrix with $UU^T=U^TU=I$ (to avoid sidetracking the flow of the presentation with constant mentions of special cases, we assume that the eigenvalues of $S$ are positive and belong to an interval independent of $d$.). Let $A\in\mR^{n\times d}$ be an $n\times d$ matrix comprised of independent standard normal elements. We can then write the following statistical equivalent to (\ref{eq:inteq1ad0})
\begin{eqnarray}\label{eq:inteq1ad2}
\lambda_n\lp \hat{\Sigma} -\Sigma   \rp 
 &  =  & \max_{\x\in\mS^d} \lp  \frac{1}{n}\x^T X^TX \x - \x^T\Sigma  \x \rp
\nonumber \\
 &  =  &
 \max_{\x\in\mS^d} \lp  \frac{1}{n}\x^T US^TU^TA^TAUSU^T \x - \x^T USS^TU^T \x \rp
\nonumber \\
 &  =  &
 \max_{\x\in\mS^d} \lp  \frac{1}{n}\x^T S^TU^TA^TAUS \x - \x^T SS^T \x \rp
\nonumber \\
 &  =  &
 \max_{\x\in\mS^d} \lp  \frac{1}{n}\x^T S^TA^TAS \x - \x^T SS^T \x \rp.
\end{eqnarray}
Since $\x\in\mS^d$, the third equality follows through the cosmetic change of variables, $U^T\x\rightarrow \x$. Similarly, the fourth equality is enabled by the rotational invariance of $A$, i.e., by the fact that $A$ and $AU$ have the same distribution. For a real scalar $c$, we find it useful to set
\begin{eqnarray}\label{eq:inteq1a}
 \xi(c) & = &   \frac{1}{\sqrt{n}} \max_{\x\in\mS^d,\|S\x\|_2=c }  \sqrt{\x^T S^TA^TAS \x}.
\end{eqnarray}
This allows to rewrite (\ref{eq:inteq1ad2}) as
\begin{eqnarray}\label{eq:inteq1ab00}
\lambda_n\lp \hat{\Sigma} -\Sigma   \rp 
  &  =  &
 \max_{\x\in\mS^d} \lp  \frac{1}{n}\x^T S^TA^TAS \x - \x^T SS^T \x \rp 
 \nonumber \\
  &  =  &
 \max_{\x\in\mS^d,\|S\x\|_2=c,c} \lp  \frac{1}{n}\x^T S^TA^TAS \x - \x^T SS^T \x \rp 
 \nonumber \\
  &  =  &
 \max_{c} \lp  \max_{\x\in\mS^d,\|S\x\|_2=c} \frac{1}{n}\x^T S^TA^TAS \x - c^2\rp 
 \nonumber \\
  &  =  &
 \max_{c} \lp \lp \xi(c)\rp^2 - c^2 \rp.
\end{eqnarray}
To characterize $\xi(c)$ and ultimately $\lambda_n\lp \hat{\Sigma} -\Sigma   \rp $, we utilize \emph{Random duality theory} (RDT). This is done in the following section.

\section{Characterizing $\xi(c)$ via RDT }
\label{sec:xirdt}

We first recall the four key RDT principles \cite{StojnicCSetam09,StojnicRegRndDlt10} (for more on further upgrades and associated algorithmic implications, see, e.g., \cite{Stojnicalgbp25,Stojnicclupsk25}):

\begin{enumerate}

\item \emph{Finding underlying optimization algebraic representation}

\item \emph{Determining random dual} 

\item \emph{Handling random dual}

\item \emph{Double-checking strong random duality.}

\end{enumerate}
We below discuss each of these principles and how they relate to the problem of our interest here.

\subsection{Finding underlying optimization algebraic representation} 
\label{sec:randpr}

We start by observing that $\xi(c)$ can be rewritten in the following  RDT favorable fashion  
\begin{eqnarray}\label{eq:inteq1ab0}
\xi(c) & = &   \frac{1}{\sqrt{n}} \max_{\x\in\mS^d,\|S\x\|_2=c } \sqrt{\x^T S^TA^TAS \x}
\nonumber \\
& = &
\frac{1}{\sqrt{n}}   \max_{\x\in\mS^d ,\|S\x\|_2=c} \| AS\x \|_2
\nonumber \\
& = &
 \frac{1}{\sqrt{n}} \max_{\x\in\mS^d,\|S\x\|_2=c,\y\in\mS^n} \y^TAS\x .
 \end{eqnarray}
Within the RDT the above represents the so-called \emph{random primal}.

\subsection{Determining random dual} 
\label{sec:randdual}

The following theorem establishes the so-called random dual.

\begin{theorem}
\label{thm:thm1}
Consider large $n,d\in\mN$ such that $\lim_{n\rightarrow\infty}\frac{n}{d}\rightarrow\alpha$ and let $A_{ij}\sim \cN(0,1)$ be the independent elements of $A\in\mR^{n\times d}$. Let $\g^{(2)}\in\mR^{d\times 1}$ be a $d$-dimensional vector comprised of independent standard normals. Assume also that $A$ and $\g^{(2)}$ are independent of each other.  For $\s\in\mR_+^{d\times 1}$, diagonal $S\in\mR^{d\times d}$ such that $S=\mbox{diag}(\s)$, and $c\in(\min(\s),\max(\s))$ set
\begin{eqnarray}
   \label{eq:thm1eq1}
 L(c) & = &  
 \max_{\x\in\mS^d,\|S\x\|_2=c} \lp\g^{(2)}\rp^TS\x.
\end{eqnarray}
Let $\xi(c)$ be as in (\ref{eq:inteq1ab0}).  One then has  
\begin{eqnarray}
   \label{eq:thm1eq2}
\mE \xi(c)  \leq c +  \frac{1}{\sqrt{n}} \mE L(c),
\end{eqnarray}
with the righthand side being the so-called random dual.
\end{theorem}

\begin{proof} Let $\g^{(1)}\in\mR^{n\times 1}$ and $g\in\mR$ be comprised of standard normals independent among themselves and of all other random variables. We consider two centered Gaussian processes indexed by an array $\cX = \{\x,\y\}$
  \begin{eqnarray}
\label{eq:mr1}
 \cG (\cX) & \triangleq &  \cG (\x,\y)  \triangleq  \sum_{i=1}^n \sum_{j=1}^m A_{i,j}\s_i\x_i\y_j  + c g  \nonumber   \\
 \cG_u (\cX) & \triangleq &  \cG_u (\x,\y)  \triangleq    c\lp\g^{(1)}\rp^T\y +   \lp\g^{(2)}\rp^TS\x .
  \end{eqnarray}
One then clearly must have $\x \in\mR^{d\times 1}$ and $\y \in\mR^{n\times 1}$. Taking two arrays $\cX^{(a_1)}=\{ \x^{(a_1)},\y^{(a_1)}\}$ and $\cX^{(a_2)}=\{ \x^{(a_2)},\y^{(a_2)}\}$ with $\|\x^{(a_i)}\|_2=\|\y^{(a_i)}\|_2=1$ and $\|S\x^{(a_i)}\|_2=c$, $i=1,2$,  we further write
  \begin{eqnarray}
\label{eq:mr2}
\mE \cG (\cX^{(a_1)})\cG (\cX^{(a_2)})   & =  &    \lp \x^{(a_1)} \rp^TS^TS\x^{(a_2)} \lp \y^{(a_1)} \rp^T\y^{(a_2)}  + c^2
\nonumber   \\
\mE \cG_u (\cX^{(a_1)})\cG_u (\cX^{(a_2)})  & =  &  c^2\lp \y^{(a_1)} \rp^T\y^{(a_2)}  +   \lp \x^{(a_1)} \rp^TS^TS\x^{(a_2)} .
  \end{eqnarray}
It is not that difficult to see that the above practically evaluates the average correlation/overlap between two  replicas for both $\cG (\cX)$ and $\cG_u (\cX)$ processes. From (\ref{eq:mr2}), one the has
  \begin{eqnarray}
\label{eq:mr5}
  \mE \cG (\cX^{(a_1)})\cG (\cX^{(a_2)})
 -
\mE \cG_u (\cX^{(a_1)})\cG_u (\cX^{(a_2)} )
 & = &    \lp \x^{(a_1)} \rp^TS^TS\x^{(a_2)} \lp \y^{(a_1)} \rp^T\y^{(a_2)}  + c^2 
\nonumber
\\
& &  - c^2\lp \y^{(a_1)} \rp^T\y^{(a_2)}  -   \lp \x^{(a_1)} \rp^TS^TS\x^{(a_2)}
\nonumber
\\
& = &
 \lp c^2- \lp \x^{(a_1)} \rp^TS^TS\x^{(a_2)}  \rp \lp 1 - \lp \y^{(a_1)} \rp^T\y^{(a_2)}\rp
  \geq   0. \nonumber \\
  \end{eqnarray}
For the completeness, we also observe  
  \begin{eqnarray}
\label{eq:mr5a0}
  \mE \cG (\cX^{(a_1)})\cG (\cX^{(a_1)})
 -
\mE \cG_u (\cX^{(a_1)})\cG_u (\cX^{(a_1)} )
  = 
 \lp c^2- \lp \x^{(a_1)} \rp^T S^TS \x^{(a_1)}  \rp \lp 1 - \lp \y^{(a_1)} \rp^T\y^{(a_1)}\rp
  =   0,
  \end{eqnarray}
  where the last equality follows since $\|S\x^{(a_i)}\|_2=c^2$ and/or $\y^{(a_i)}\in\mS^m$ for $i=1,2$.
  
Digressing for a moment, we recall on Theorem 1.1 from \cite{Gordon85} (the part of the theorem stated below is actually known as Slepian lemma and is introduced earlier in \cite{Slep62}; both results are special cases of concepts discussed in Corollary 3 in \cite{Stojnicgscompyx16}  and in Corollary 4 in \cite{Stojnicgscomp16}).

\begin{theorem}(\cite{Gordon85,Slep62})
\label{thm:Gordonpos1} Let $X_{i}$ and $Y_{i}$, $1\leq i\leq n$, be two centered Gaussian processes which satisfy the following inequalities for all choices of indices
\begin{enumerate}
\item $\mE(X_{i}^2)=\mE(Y_{i}^2)$
\item $\mE(X_{i}X_{l})\leq \mE(Y_{i}Y_{l}), i\neq l$.
\end{enumerate}
 Then
\begin{equation*}
\mE(\min_{i} X_{i})\leq \mE(\min_i Y_{i}) \quad  \Longleftrightarrow \quad \mE(\max_{i} X_{i})\geq \mE(\max_i Y_{i}).
\end{equation*}
\end{theorem}

Noting correspondence $Y\leftrightarrow\cG$ and $X\leftrightarrow\cG_u$, allows us to apply  Theorem \ref{thm:Gordonpos1} to processes $\cG(\cdot)$ and  $\cG_u(\cdot)$. As a result we have
\begin{align}\label{eq:mt5a1a0}
&  &\mE \max_{\cX^{(a_1)}} \cG(\cX)  & \leq \mE \max_{\cX^{(a_1)}} \cG_u(\cX)
\nonumber \\
\Longleftrightarrow & & \mE \max_{\x\in\mS^d,\|S\x\|_2=c,\y\in\mS^n} \lp \sum_{i=1}^n \sum_{j=1}^m A_{i,j}\s_i\x_i\y_j  + c g\rp & \leq \mE \max_{\x\in\mS^d,\|S\x\|_2=c,\y\in\mS^n}  \lp c\lp\g^{(1)}\rp^T\y +   \lp\g^{(2)}\rp^TS\x \rp
\nonumber \\
\Longleftrightarrow & & \mE \max_{\x\in\mS^d,\|S\x\|_2=c,\y\in\mS^n}  \y^TAS\x & \leq \mE \max_{\x\in\mS^d,\|S\x\|_2=c}  \lp\| c\g^{(1)}\|_2 +   \lp\g^{(2)}\rp^TS\x \rp.
 \end{align}
Connecting further (\ref{eq:inteq1ab0}) and (\ref{eq:mt5a1a0}), we then also find
\begin{eqnarray}\label{eq:mt5a1a1}
 \mE  \xi(c)  & \leq &  \frac{c}{\sqrt{n}} \mE\| \g^{(1)}\|_2  +  \mE \max_{\x\in\mS^d,\|S\x\|_2=c }     \lp\g^{(2)}\rp^TS\x
 \nonumber \\
 & \leq &  \frac{c}{\sqrt{n}} \sqrt{\mE\| \g^{(1)}\|_2^2}  +  \frac{1}{\sqrt{n}}  \mE \max_{\x\in\mS^d,\|S\x\|_2=c }     \lp\g^{(2)}\rp^TS\x
  \nonumber \\
 & \leq &  c +  \frac{1}{\sqrt{n}}  \mE \max_{\x\in\mS^d,\|S\x\|_2=c}     \lp\g^{(2)}\rp^TS\x,
 \end{eqnarray}
which, together with (\ref{eq:thm1eq1}), gives  (\ref{eq:thm1eq2}) and completes the proof.
\end{proof}

\begin{remark}
\label{rem:rem0}
  To make the presentation neater and writing easier, we throughout the paper focus on expectations. However, all key statistical objects that we consider trivially concentrate and the results hold in probabilistic sense as well. We skip emphasizing these facts as we progress through the presentation.
\end{remark}

We also note that the first two main RDT principles were rewritten in a generic processes comparisons context  in \cite{qihan24} and the sample covariance error with the bound from (\ref{eq:mt5a1a1}) was considered as a particular application.

\subsection{Handling random dual}
\label{sec:handlerd}

To handle the above random dual, we closely follow the paths traced in \cite{StojnicCSetam09,StojnicICASSP10var,StojnicISIT2010binary}.  Before proceeding with the details, we find it useful to note the role of the random dual with the overall RDT mosaic. Namely, a combination of (\ref{eq:inteq1ab00}), (\ref{eq:thm1eq1}), and (\ref{eq:thm1eq2}) together with concentrations gives
\begin{eqnarray}\label{eq:hrd0}
\mE \lambda_n\lp \hat{\Sigma} -\Sigma   \rp 
   &  =  &
 \max_{c} \lp \lp \mE\xi \rp^2 - c^2 \rp
 \nonumber \\
    &  \leq  &
   \max_{c} \lp \lp c + \mE \frac{1}{\sqrt{n}} \max_{\x\in\mS^d,\|S\x\|_2=c} \lp\g^{(2)}\rp^TS\x \rp^2 - c^2 \rp 
 \nonumber \\
    &  =  &
 \max_{c} \lp \lp c + \frac{1}{\sqrt{n}} \mE L(c) \rp^2 - c^2 \rp .
\end{eqnarray}
 Clearly, $L(c)$ is the key object of interest and we below study it in detail. First, we note 
\begin{eqnarray}
   \label{eq:hrd1}
 L(c) & \triangleq &  
  \max_{\x\in\mS^d,\|S\x\|_2=c} \lp\g^{(2)}\rp^TS\x
  = \max_{\x\in\mS^d,\|S\x\|_2=c} \sum_{i=1}^{d}\g_i^{(2)} \s_i\x_i
  = -\min_{\x\in\mS^d,\|S\x\|_2=c} \sum_{i=1}^{d}\g_i^{(2)} \s_i\x_i .
\end{eqnarray}
Then we have for the Lagrangian 
\begin{eqnarray}
   \label{eq:hrd2}
 \cL  = \sum_{i=1}^{d}\g_i^{(2)} \s_i\x_i + \gamma \sum_{i=1}^{d}\x_i^2 -\gamma + \gamma_0 \sum_{i=1}^{d}\s_i^2\x_i^2 -\gamma_0 c^2.
\end{eqnarray}
Combining (\ref{eq:hrd1}) and  (\ref{eq:hrd2}) and utilizing strong duality we further write
\begin{eqnarray}
   \label{eq:hrd3}
 L(c)  = -\min_{\x}\max_{\gamma,\gamma_0} \cL= -\max_{\gamma,\gamma_0} \min_{\x} \cL.
\end{eqnarray}
To optimize over $\x$ we first find derivatives
\begin{eqnarray}
   \label{eq:hrd4}
 \frac{d\cL}{d\x_i}  =  \g_i^{(2)} \s_i + 2\gamma \x_i + 2\gamma_0 \s_i^2\x_i.
\end{eqnarray}
Equalling the above derivatives to zero gives
\begin{eqnarray}
   \label{eq:hrd5}
 \x_i = -\frac{\g_i^{(2)} \s_i}{2(\gamma + \gamma_0 \s_i^2)}.
\end{eqnarray}
Plugging this back in (\ref{eq:hrd2}), one finds
\begin{eqnarray}
   \label{eq:hrd6}
\min_{\x} \cL = -\frac{1}{4}\sum_{i=1}^{d}\frac{\lp\g_i^{(2)}\rp^2 \s_i^2}{\gamma + \gamma_0 \s_i^2} 
 -\gamma  -\gamma_0 c^2.
\end{eqnarray}
A combination of (\ref{eq:hrd3}) and  (\ref{eq:hrd6}) gives
\begin{eqnarray}
   \label{eq:hrd7}
 L(c)  & = & -\max_{\gamma,\gamma_0} \min_{\x} \cL
 \nonumber \\
& = & -\max_{\gamma,\gamma_0} 
 \lp    -\frac{1}{4}\sum_{i=1}^{d}\frac{\lp\g_i^{(2)}\rp^2 \s_i^2}{\gamma  + \gamma_0 \s_i^2} 
 -\gamma  -\gamma_0 c^2
 \rp
  \nonumber \\
& = & \min_{\gamma,\gamma_0} 
 \lp    \frac{1}{4}\sum_{i=1}^{d}\frac{\lp\g_i^{(2)}\rp^2 \s_i^2}{\gamma  + \gamma_0 \s_i^2} 
 +\gamma  +\gamma_0 c^2
 \rp
  \nonumber \\
& = & \min_{\gamma_x,\gamma_0} 
 \lp    \frac{1}{4\gamma_0}\sum_{i=1}^{d}\frac{\lp\g_i^{(2)}\rp^2 \s_i^2}{\gamma_x  + \s_i^2} 
 + \gamma_0(\gamma_x  + c^2)
 \rp
  \nonumber \\
& = & \min_{\gamma_x} 
 \sqrt{\sum_{i=1}^{d}\frac{\lp\g_i^{(2)}\rp^2 \s_i^2}{\gamma_x + \s_i^2} 
 (\gamma_x  + c^2) }
\end{eqnarray}
Law of large numbers and concentrations gives
\begin{eqnarray}
   \label{eq:hrd8}
\lim_{d\rightarrow\infty} \frac{1}{\sqrt{d}} \mE L(c)  & = & \lim_{d\rightarrow\infty} \sqrt{\min_{\gamma_x} \bar{L}},
\end{eqnarray}
where
\begin{eqnarray}
   \label{eq:hrd9}
\bar{L} = 
  \lp \frac{1}{d}\sum_{i=1}^{d}\frac{  \s_i^2}{\gamma_x  + \s_i^2} 
 (\gamma_x  + c^2) \rp.
\end{eqnarray}
Combining  (\ref{eq:hrd0}), (\ref{eq:hrd8}), and  (\ref{eq:hrd9}), we find
\begin{eqnarray}\label{eq:hrd10}
\lim_{d\rightarrow \infty}\mE \lambda_n\lp \hat{\Sigma} -\Sigma   \rp 
     &  \leq  &
 \max_{c} \lp \lp c + \lim_{d\rightarrow \infty } \frac{1}{\sqrt{n}}\mE L(c)   \rp^2 - c^2 \rp 
 =
\lim_{d\rightarrow\infty}   \max_{c}\min_{\gamma_x}  \delta_u(\alpha)  ,
\end{eqnarray}
where
\begin{eqnarray}\label{eq:hrd11}
\delta_u(\alpha)  \triangleq  \lp c +  \frac{1}{\sqrt{\alpha}}\sqrt{ \bar{L}}   \rp^2 - c^2  = 2 \frac{c\sqrt{\bar{L}}}{\sqrt{\alpha}}  
+
 \frac{\bar{L}}{\alpha} .
\end{eqnarray}
Computing $\gamma_x$ derivative gives
\begin{eqnarray}
   \label{eq:hrd12}
\frac{d\bar{L}}{d\gamma_x} = 
   -\frac{1}{d}\sum_{i=1}^{d}\frac{  \s_i^2}{(\gamma_x  + \s_i^2)^2} 
 (\gamma_x  + c^2) 
+ \frac{1}{d}\sum_{i=1}^{d}\frac{  \s_i^2}{\gamma_x  + \s_i^2} 
=
   \frac{1}{d}\sum_{i=1}^{d}\frac{ \s_i^2( \s_i^2 -c^2)}{(\gamma_x  + \s_i^2)^2} .
\end{eqnarray}
Computing $c$ derivative gives
\begin{eqnarray}
   \label{eq:hrd13}
\frac{d\delta_u(\alpha)}{d c} 
& = & 
 \frac{2\sqrt{\bar{L}}}{\sqrt{\alpha}} + \frac{c}{\sqrt{\alpha \bar{L}}}  
  \lp \frac{2c}{d}\sum_{i=1}^{d}\frac{  \s_i^2}{\gamma_x  + \s_i^2} 
  \rp
 +
    \frac{2}{\alpha d}\sum_{i=1}^{d}\frac{  \s_i^2}{\gamma_x  + \s_i^2} 
\nonumber \\
&  = &
\frac{2}{\sqrt{\alpha}}
\lp
 \frac{\bar{L}}{\sqrt{\bar{L}}} + \frac{c}{\sqrt{ \bar{L}}}  
  \lp \frac{c}{d}\sum_{i=1}^{d}\frac{  \s_i^2}{\gamma_x  + \s_i^2} 
  \rp
 +
    \frac{c}{\sqrt{\alpha} d}\sum_{i=1}^{d}\frac{  \s_i^2}{\gamma_x  + \s_i^2}  
\rp   
\nonumber \\
&  = &
\frac{2}{\sqrt{\alpha}}
\lp
 \frac{1}{\sqrt{\bar{L}}} \lp \frac{1}{d}\sum_{i=1}^{d}\frac{  \s_i^2}{\gamma_x  + \s_i^2} 
 (\gamma_x  + c^2) \rp + \frac{c}{\sqrt{ \bar{L}}}  
  \lp \frac{c}{d}\sum_{i=1}^{d}\frac{  \s_i^2}{\gamma_x  + \s_i^2} 
  \rp
 +
    \frac{c}{\sqrt{\alpha} d}\sum_{i=1}^{d}\frac{  \s_i^2}{\gamma_x  + \s_i^2}  
\rp   
\nonumber \\
&  = &
\frac{2}{\sqrt{\alpha}}\frac{1}{d}\sum_{i=1}^{d}\frac{  \s_i^2}{\gamma_x  + \s_i^2}
\lp
 \frac{\gamma_x  + c^2 }{\sqrt{\bar{L}}}  + \frac{c^2}{\sqrt{ \bar{L}}}  
 +
    \frac{c}{\sqrt{\alpha}}  
\rp   
\nonumber \\
&  = &
\frac{2}{\sqrt{\alpha}}\frac{1}{d}\sum_{i=1}^{d}\frac{  \s_i^2}{\gamma_x  + \s_i^2}
\lp
 \frac{\gamma_x  + 2c^2 }{\sqrt{\bar{L}}} 
  +
    \frac{c}{\sqrt{\alpha}}  
\rp.   
 \end{eqnarray}
Equalling the above derivative to zero gives
\begin{eqnarray}
   \label{eq:hrd14}
\lp
 \frac{\gamma_x  + 2c^2 }{\sqrt{\bar{L}}} 
  +
    \frac{c}{\sqrt{\alpha}}  
\rp = 0 \quad \quad \Longrightarrow \quad \quad  \alpha\lp \gamma_x  + 2c^2 \rp^2 = c^2\bar{L}
   = c^2  z 
 (\gamma_x  + c^2) ,
  \end{eqnarray}
where
\begin{eqnarray}
   \label{eq:hrd15}
 z =  \frac{1}{d}\sum_{i=1}^{d}\frac{  \s_i^2}{\gamma_x  + \s_i^2} .
  \end{eqnarray}
Keeping in mind $c\geq 0$, we also observe that equalling the last term in (\ref{eq:hrd13}) to zero, additionally implies 
\begin{eqnarray}
   \label{eq:hrd15a0}
\lp
 \frac{\gamma_x  + 2c^2 }{\sqrt{\bar{L}}} 
  +
    \frac{c}{\sqrt{\alpha}}  
\rp = 0 \quad \quad \Longrightarrow \quad \quad   
\gamma_x\leq 0 \quad  \mbox{and }  \gamma_x+ 2 c^2\leq 0 .
  \end{eqnarray}
One then finds the following equivalent to (\ref{eq:hrd14})
\begin{eqnarray}
   \label{eq:hrd16}
c^4  +c^2\gamma_x -\frac{\gamma_x^2\alpha}{z-4\alpha} = 0.
  \end{eqnarray}
Solving for $c^2$ gives
\begin{eqnarray}
   \label{eq:hrd17}
c^2 = \frac{-\gamma_x \pm \sqrt{ \gamma_x^2  +4\frac{\gamma_x^2\alpha}{z-4\alpha}  }}{2}
= \gamma_x\frac{- 1 \pm \mbox{sign}(\gamma_x) \sqrt{ 1  +4\frac{\alpha}{z-4\alpha}  }}{2}
= \gamma_x\frac{- 1 \pm \mbox{sign}(\gamma_x) \sqrt{ \frac{z}{z-4\alpha}  }}{2}.
  \end{eqnarray}
 Taking into account (\ref{eq:hrd15a0}), we then have particular choice of signs that gives
\begin{eqnarray}
   \label{eq:hrd17a0}
c^2  = \gamma_x\frac{- 1 + \sqrt{ \frac{z}{z-4\alpha}  }}{2}.
  \end{eqnarray}
  
We first set
\begin{eqnarray}
   \label{eq:hrd17a1}
 \phi_1(\gamma_x)  \triangleq  \frac{1}{d}\sum_{i=1}^{d}\frac{  \s_i^4}{\lp \gamma_x  + \s_i^2 \rp^2}, 
 \quad \quad
 \phi_2(\gamma_x)  \triangleq  \frac{1}{d}\sum_{i=1}^{d}\frac{  \s_i^2}{\lp\gamma_x  + \s_i^2\rp^2}, 
   \end{eqnarray}
and find
\begin{eqnarray}
   \label{eq:hrd18}
  z & = &   \frac{1}{d}\sum_{i=1}^{d}\frac{  \s_i^2}{\gamma_x  + \s_i^2} 
  = \gamma_x\phi_2(\gamma_x) + \phi_1(\gamma_x)
  \nonumber \\
  \bar{L}  & = &  
   \frac{1}{d}\sum_{i=1}^{d}\frac{  \s_i^2}{\gamma_x  + \s_i^2} 
 (\gamma_x  + c^2) = z(\gamma_x  + c^2) = (\gamma_x\phi_2(\gamma_x) + \phi_1(\gamma_x))(\gamma_x  + c^2).
  \end{eqnarray}
After additionally setting
\begin{eqnarray}
   \label{eq:hrd19}
\phi_3(\gamma_x)  
= \gamma_x\frac{- 1 + \sqrt{ \frac{z}{z-4\alpha}  }}{2}
= \gamma_x\frac{- 1  + \sqrt{ \frac{\gamma_x\phi_2(\gamma_x) + \phi_1(\gamma_x)}{\gamma_x\phi_2(\gamma_x) + \phi_1(\gamma_x)-4\alpha}  }}{2},
  \end{eqnarray}
the following theorem summarize handling random dual.

\begin{theorem}
\label{thm:thm3}
Assume the setup of Theorem \ref{thm:thm1} with $\s$ and $\alpha$ such that all quantities below are in $\mR$ and $\hat{\delta}_u(\alpha) \geq 0$. For $\phi_1(\cdot)$, $\phi_2(\cdot)$, and $\phi_3(\cdot)$ from (\ref{eq:hrd17a1}) and (\ref{eq:hrd19}),  let $\hat{\gamma}_x$ satisfy
\begin{eqnarray}
   \label{eq:thm3eq1}
\lim_{d\rightarrow\infty} \lp \phi_1(\hat{\gamma}_x) -\phi_2(\hat{\gamma}_x)\phi_3(\hat{\gamma}_x) \rp = 0
 \quad \quad 
\Longleftrightarrow 
 \quad \quad 
\lim_{d\rightarrow\infty}  \lp  \hat{\gamma}_x \phi_2(\hat{\gamma}_x) - \phi_1(\hat{\gamma}_x) \frac{ 2\sqrt{\alpha}-\sqrt{\phi_1(\hat{\gamma}_x) } } {\sqrt{\phi_1(\hat{\gamma}_x)}-\sqrt{\alpha}} \rp  =0.
\end{eqnarray}
Also, let 
\begin{eqnarray}
   \label{eq:thm3eq1a0}
\hat{c} = \lim_{d\rightarrow\infty}  \sqrt{\phi_3(\hat{\gamma}_x)} 
= \lim_{d\rightarrow\infty}  \sqrt{ \frac{ \phi_1(\hat{\gamma}_x) }  { \phi_2(\hat{\gamma}_x) }  } 
=  
\lim_{d\rightarrow\infty}  \sqrt{\frac{\hat{\gamma}_x (\sqrt{\phi_1(\hat{\gamma}_x)} -\sqrt{\alpha}) } { 2\sqrt{\alpha}- \sqrt{\phi_1(\hat{\gamma}_x)}  }} .
\end{eqnarray}
One then has  
\begin{equation}\label{eq:thm3eq2}
\lim_{d\rightarrow \infty}\mE \lambda_n\lp \hat{\Sigma} -\Sigma   \rp 
       \leq  \lim_{d\rightarrow\infty}  \max_{ c } \min_{ \gamma_x}    \delta_u(\alpha) = \hat{\delta}_u(\alpha) ,
       \end{equation}
       where
\begin{equation}\label{eq:thm3eq3}
\hat{\delta}_u(\alpha) =         \lim_{d\rightarrow\infty} 
\frac{\hat{\gamma}_x\sqrt{\phi_1(\hat{\gamma}_x)}} {\sqrt{\phi_1(\hat{\gamma}_x)} -\sqrt{\alpha} } .
  \end{equation}
\end{theorem}

\begin{proof} From (\ref{eq:hrd9})-(\ref{eq:hrd11}) and the above discussion, we first have for $\gamma_x=\hat{\gamma}_x$ and $c=\hat{c}$
\begin{eqnarray}
\label{eq:prthm31}
\delta_u(\alpha)  & = &    2 \frac{c\sqrt{\bar{L}}}{\sqrt{\alpha}}  
+
 \frac{\bar{L}}{\alpha}  
 \nonumber \\
   &=&  
   \lp 
   \frac{2}{\sqrt{\alpha}} \sqrt{\phi_3(\hat{\gamma}_x)} \sqrt{(\hat{\gamma}_x\phi_2(\hat{\gamma}_x) + \phi_1(\hat{\gamma}_x))(\hat{\gamma}_x  + \phi_3(\hat{\gamma}_x)) } + \frac{(\hat{\gamma}_x\phi_2(\hat{\gamma}_x) + \phi_1(\hat{\gamma}_x))(\hat{\gamma}_x  + \phi_3(\hat{\gamma}_x) )}{\alpha} \rp 
    \nonumber \\
   &=& 
      \lp 
   -\frac{2}{\sqrt{\alpha}} \sqrt{\frac{ \phi_3(\hat{\gamma}_x) }{ \phi_2(\hat{\gamma}_x)   }  } (\hat{\gamma}_x\phi_2(\hat{\gamma}_x) + \phi_1(\hat{\gamma}_x)) + \frac{(\hat{\gamma}_x\phi_2(\hat{\gamma}_x) + \phi_1(\hat{\gamma}_x))^2}{\alpha \phi_2(\hat{\gamma}_x) } \rp 
   \nonumber \\
   &=& 
      \lp 
   -\frac{2}{\sqrt{\alpha}} \sqrt{\frac{ \phi_1(\hat{\gamma}_x) }{ (\phi_2(\hat{\gamma}_x) )^2  }  } (\hat{\gamma}_x\phi_2(\hat{\gamma}_x) + \phi_1(\hat{\gamma}_x)) + \frac{(\hat{\gamma}_x\phi_2(\hat{\gamma}_x) + \phi_1(\hat{\gamma}_x))^2}{\alpha \phi_2(\hat{\gamma}_x) } \rp  
   \nonumber \\
   &=& 
       \frac{\hat{\gamma}_x\phi_2(\hat{\gamma}_x) + \phi_1(\hat{\gamma}_x)}{\alpha \phi_2(\hat{\gamma}_x) } 
      \lp 
    -2\sqrt{\alpha \phi_1(\hat{\gamma}_x)  } + \hat{\gamma}_x\phi_2(\hat{\gamma}_x) + \phi_1(\hat{\gamma}_x)   \rp
       \nonumber \\
   &=& 
       \frac{z}{\alpha \phi_2(\hat{\gamma}_x) } 
      \lp 
    -2\sqrt{\alpha \phi_1(\hat{\gamma}_x)  } + z   \rp
          \nonumber \\
   &=& 
       \frac{\gamma_x z}{\alpha (z- \phi_1(\hat{\gamma}_x)) } 
      \lp 
    -2\sqrt{\alpha \phi_1(\hat{\gamma}_x)  } + z   \rp,
\end{eqnarray}
  where we utilize $z$ from (\ref{eq:hrd18}). We then further note that $z$ can be expressed as a function of $\phi_1(\hat{\gamma}_x)$. Namely, from (\ref{eq:hrd19}), one has
\begin{align}
   \label{eq:prfthm32}
&  & \phi_3(\gamma_x)    
& =  \gamma_x\frac{- 1  + \sqrt{ \frac{\gamma_x\phi_2(\gamma_x) + \phi_1(\gamma_x)}{\gamma_x\phi_2(\gamma_x) + \phi_1(\gamma_x)-4\alpha}  }}{2}
\nonumber \\
\Longleftrightarrow  & & 
2\phi_1(\gamma_x)   &  
=  - \gamma_x\phi_2(\gamma_x)   + \gamma_x\phi_2(\gamma_x) \sqrt{ \frac{\gamma_x\phi_2(\gamma_x) + \phi_1(\gamma_x)}{\gamma_x\phi_2(\gamma_x) + \phi_1(\gamma_x)-4\alpha}  }
\nonumber \\
\Longleftrightarrow  & & 
z+ \phi_1(\gamma_x)   &  
=  (z- \phi_1(\gamma_x)  )  \sqrt{ \frac{ z}{z-4\alpha}  },
  \end{align}
  where we again utilize $z$ from (\ref{eq:hrd18}).  Solving over $z$ gives
\begin{eqnarray}
   \label{eq:prfthm33}
z = \frac{\alpha \phi_1(\hat{\gamma}_x) + \phi_1(\hat{\gamma}_x) \sqrt{\alpha \phi_1(\hat{\gamma}_x)  }}{\phi_1(\hat{\gamma}_x)-\alpha} 
= \frac{ \phi_1(\hat{\gamma}_x) \sqrt{\alpha} } {\sqrt{\phi_1(\hat{\gamma}_x)}-\sqrt{\alpha}} .
\end{eqnarray}
With $z$ from  (\ref{eq:hrd18}), we also have
\begin{align}
   \label{eq:prfthm33a0}
 & & \gamma_x \phi_2(\hat{\gamma}_x) + \phi_1(\hat{\gamma}_x)  &   
= \frac{ \phi_1(\hat{\gamma}_x) \sqrt{\alpha} } {\sqrt{\phi_1(\hat{\gamma}_x)}-\sqrt{\alpha}} 
\nonumber \\
\Longleftrightarrow & & \gamma_x \phi_2(\hat{\gamma}_x) + \phi_1(\hat{\gamma}_x)  - \frac{ \phi_1(\hat{\gamma}_x) \sqrt{\alpha} } {\sqrt{\phi_1(\hat{\gamma}_x)}-\sqrt{\alpha}} & =0
\nonumber \\
\Longleftrightarrow & & \gamma_x \phi_2(\hat{\gamma}_x) - \phi_1(\hat{\gamma}_x) \frac{ 2\sqrt{\alpha}- \sqrt{ \phi_1(\hat{\gamma}_x) } } {\sqrt{\phi_1(\hat{\gamma}_x)}-\sqrt{\alpha}} & =0,
\end{align}
which matches the righthand side condition in (\ref{eq:thm3eq1}) as well as the last equality in  (\ref{eq:thm3eq1a0}). Finally, plugging $z$ back in (\ref{eq:prthm31}) also gives
\begin{eqnarray}
\label{eq:prthm34}
\delta_u(\alpha)  
& = &     
       \frac{\gamma_x \frac{ \phi_1(\hat{\gamma}_x) \sqrt{\alpha} } {\sqrt{\phi_1(\hat{\gamma}_x)}-\sqrt{\alpha}}}{\alpha \lp \frac{ \phi_1(\hat{\gamma}_x) \sqrt{\alpha} } {\sqrt{\phi_1(\hat{\gamma}_x)}-\sqrt{\alpha}}- \phi_1(\hat{\gamma}_x) \rp } 
      \lp 
    -2\sqrt{\alpha \phi_1(\hat{\gamma}_x)  } + \frac{ \phi_1(\hat{\gamma}_x) \sqrt{\alpha} } {\sqrt{\phi_1(\hat{\gamma}_x)}-\sqrt{\alpha}}   \rp
  \nonumber \\
  & = &     
       \frac{\gamma_x \frac{ \sqrt{\phi_1(\hat{\gamma}_x)}  } {\sqrt{\phi_1(\hat{\gamma}_x)}-\sqrt{\alpha}}}{  \frac{  \sqrt{\alpha} } {\sqrt{\phi_1(\hat{\gamma}_x)}-\sqrt{\alpha}}- 1  } 
      \lp 
    -2  + \frac{ \sqrt{\phi_1(\hat{\gamma}_x)} } {\sqrt{\phi_1(\hat{\gamma}_x)}-\sqrt{\alpha}}   \rp  
  \nonumber \\
  & = &     
       \frac{\gamma_x \frac{ \sqrt{\phi_1(\hat{\gamma}_x)}  } {\sqrt{\phi_1(\hat{\gamma}_x)}-\sqrt{\alpha}}}{  \frac{  \sqrt{\alpha} } {\sqrt{\phi_1(\hat{\gamma}_x)}-\sqrt{\alpha}}- 1  } 
      \lp 
    -1  + \frac{ \sqrt{\alpha} } {\sqrt{\phi_1(\hat{\gamma}_x)}-\sqrt{\alpha}}   \rp  
  \nonumber \\
  & = &     
        \gamma_x \frac{ \sqrt{\phi_1(\hat{\gamma}_x)}  } {\sqrt{\phi_1(\hat{\gamma}_x)}-\sqrt{\alpha}} ,       
\end{eqnarray}
which matches (\ref{eq:thm3eq3}) and completes the proof.
\end{proof}

\begin{remark}
\label{rem:rem1}
It should be noted that $\max_{\hat{c}} \min_{\hat{\gamma}_x} $ in (\ref{eq:thm3eq2}) is added for the completeness (just in case there are multiple solutions to (\ref{eq:thm3eq1})). Otherwise, (\ref{eq:thm3eq1}) provides full characterization of the optimal $\gamma_x$ and, via $c=\lim_{d\rightarrow\infty}  \sqrt{\phi_3(\hat{\gamma}_x)} $, of the optimal $c$. Also, constraining $\s$ and $\alpha$ is added for esthetic reasons to avoid sidetracking presentation with analyses of special cases that bring no conceptual novelty.
\end{remark}

\begin{remark}
\label{rem:rem1a}
Even though the above analyses is not well tailored for $S=I$ scenario (uncorrelated Gaussians), it manages to capture it. To see this, one observes that, for $S=I$, (\ref{eq:hrd17a1})  gives
\begin{eqnarray}
   \label{eq:rem1ahrd17a1}
 \phi_1(\gamma_x)  \triangleq  \frac{1}{d}\sum_{i=1}^{d}\frac{  \s_i^4}{\lp \gamma_x  + \s_i^2 \rp^2}
 =\frac{1}{ (\gamma_x+1)^2} 
 =   \frac{1}{d}\sum_{i=1}^{d}\frac{  \s_i^2}{\lp\gamma_x  + \s_i^2\rp^2}
 =  \phi_2(\gamma_x). 
   \end{eqnarray}
From the second condition in (\ref{eq:thm3eq1}), one then has
\begin{eqnarray}
   \label{eq:rem1athm3eq1}
 \lim_{d\rightarrow\infty}  \lp  \hat{\gamma}_x \phi_2(\hat{\gamma}_x) - \phi_1(\hat{\gamma}_x) \frac{ 2\sqrt{\alpha}-\sqrt{\phi_1(\hat{\gamma}_x) } } {\sqrt{\phi_1(\hat{\gamma}_x)}-\sqrt{\alpha}} \rp  =0.
\quad 
\Longleftrightarrow
\quad
 \lim_{d\rightarrow\infty}  \lp \sqrt{\phi_1(\hat{\gamma}_x)} (\hat{\gamma}_x +1) - (\hat{\gamma}_x +2)\sqrt{\alpha} \rp =0.
\end{eqnarray}
A combination of  (\ref{eq:rem1ahrd17a1}) and (\ref{eq:rem1athm3eq1}) gives
\begin{eqnarray}
   \label{eq:rem1athm3eq1a0}
\frac{\hat{\gamma}_x +1}{|\hat{\gamma}_x +1|} = (\hat{\gamma}_x +2)\sqrt{\alpha} .
\end{eqnarray}
By (\ref{eq:hrd15a0}), $\hat{\gamma}_x\leq -2c^2=-2$ and we further find
\begin{eqnarray}
   \label{eq:rem1athm3eq1a1}
 -\frac{1}{\sqrt{\alpha}} -2 = \hat{\gamma}_x .
\end{eqnarray}
From (\ref{eq:thm3eq3}) and (\ref{eq:rem1athm3eq1a1}), one then obtains
\begin{equation}\label{eq:rem1athm3eq3}
\hat{\delta}_u(\alpha) =        
\frac{\hat{\gamma}_x } {1-\sqrt{\alpha} |\hat{\gamma}_x +1| } 
=
\frac{-\frac{1}{\sqrt{\alpha}} -2}{1-\sqrt{\alpha} | -\frac{1}{\sqrt{\alpha}} -1 | }
=\frac{2}{\sqrt{\alpha}} +\frac{1}{\alpha} =  \frac{1}{\alpha}\lp \sqrt{\alpha} +1 \rp^2  -1,
  \end{equation}
where the most righthand side is the leading Wishart eigenvalue minus one which is exactly what one should get in the so-called isotropic case.
\end{remark}

Remark \ref{rem:rem1a} suggests that the above analysis and the resulting bounds given in Theorem \ref{thm:thm3} might be tight. We discuss this in more detail next.

\subsection{Double-checking strong random duality}
\label{sec:strrd}

Theorem \ref{thm:thm1} established the random dual which upper-bounds $\mE\xi(c)$. Below we discuss complementary lower bounds.

\subsubsection{Bilinear-quadratic mechanism}
\label{sec:bilinquad}

The following theorem establishes complementary random dual which lower-bounds $\mE \xi(c)$. It relies on a bilinear-quadratic comparative mechanism and is the key component that enables the whole machinery developed in the paper to work.

\begin{theorem}
\label{thm:strthm1}
Consider large $n,d\in\mN$ such that $\lim_{n\rightarrow\infty}\frac{n}{d}\rightarrow\alpha$ and let the elements of $A\in\mR^{n\times d}$ and  $G^{(2)}\in\mR^{d\times d}$ be independent standard normal ($A$ and $G^{(2)}$ are independent of each other as well). For $\s\in\mR_+^{d\times 1}$, diagonal $S\in\mR^{d\times d}$ such that $S=\mbox{diag}(\s)$, and $c\in(\min(\s),\max(\s))$, set
\begin{eqnarray}
   \label{eq:strthm1eq1}
 B(c) & = &  
 \max_{\x\in\mS^d,\|S\x\|_2=c} \x^TS^T G^{(2)}S\x.
\end{eqnarray}
Let $\xi(c)$ be as in (\ref{eq:inteq1ab0}).  One then has  
\begin{eqnarray}
   \label{eq:strthm1eq2}
 \lim_{n\rightarrow \infty } \mE \xi(c)  \geq c +   \lim_{n\rightarrow \infty } \frac{1}{\sqrt{2n}c} \mE B(c),
\end{eqnarray}
with the righthand side being the complementary random dual.
\end{theorem}

\begin{proof} Let $G^{(1)}\in\mR^{n\times n}$ be comprised of standard normals independent among themselves and of all other random variables. We consider two centered Gaussian processes indexed by an array $\cX = \{\x,\y\}$
  \begin{eqnarray}
\label{eq:strmr1}
 \cG_1 (\cX) & \triangleq &  \cG (\x,\y)  \triangleq  \sum_{i=1}^n \sum_{j=1}^m A_{i,j}\s_i\x_i\y_j    \nonumber   \\
 \cG_l (\cX) & \triangleq &  \cG_l (\x,\y)  \triangleq    \frac{c}{\sqrt{2}}\y^T G^{(1)} \y +   \frac{1}{\sqrt{2} c }\x^T S^T G^{(2)} S\x .
  \end{eqnarray}
We take two arrays $\cX^{(a_1)}=\{ \x^{(a_1)},\y^{(a_1)}\}$ and $\cX^{(a_2)}=\{ \x^{(a_2)},\y^{(a_2)}\}$ with $\|\x^{(a_i)}\|_2=\|\y^{(a_i)}\|_2=1$ and $\|S\x^{(a_i)}\|_2=c$, $i=1,2$ and write
  \begin{eqnarray}
\label{eq:strmr2}
\mE \cG_1 (\cX^{(a_1)})\cG_1 (\cX^{(a_2)})   & =  &    \lp \x^{(a_1)} \rp^TS^TS\x^{(a_2)} \lp \y^{(a_1)} \rp^T\y^{(a_2)}   
\nonumber   \\
\mE \cG_u (\cX^{(a_1)})\cG_u (\cX^{(a_2)})  & =  &  \frac{c^2}{2} \lp \lp \y^{(a_1)} \rp^T\y^{(a_2)}\rp^2  +   \frac{1}{2c^2}  \lp \lp \x^{(a_1)} \rp^TS^TS\x^{(a_2)} \rp^2.
  \end{eqnarray}
Combining equations from (\ref{eq:strmr2}), we obtain
  \begin{eqnarray}
\label{eq:strmr5}
  \mE \cG_1 (\cX^{(a_1)})\cG_1 (\cX^{(a_2)})
 -
\mE \cG_u (\cX^{(a_1)})\cG_u (\cX^{(a_2)} )
 & = &    \lp \x^{(a_1)} \rp^TS^TS\x^{(a_2)} \lp \y^{(a_1)} \rp^T\y^{(a_2)}  
\nonumber
\\
& &  - \frac{c^2}{2} \lp \lp \y^{(a_1)} \rp^T\y^{(a_2)}\rp^2  -   \frac{1`}{2c^2} \lp \lp \x^{(a_1)} \rp^TS^TS\x^{(a_2)} \rp^2
\nonumber
\\
& = &
 -\lp   \frac{c}{\sqrt{2}}\lp \y^{(a_1)} \rp^T\y^{(a_2)}  -  \frac{1}{\sqrt{2}c} \lp \x^{(a_1)} \rp^TS^TS\x^{(a_2)}     \rp^2
  \leq   0. \nonumber \\
  \end{eqnarray}
We also observe  
  \begin{eqnarray}
\label{eq:strmr5a0}
  \mE \cG_1 (\cX^{(a_1)})\cG_1 (\cX^{(a_1)})
 -
\mE \cG_u (\cX^{(a_1)})\cG_u (\cX^{(a_1)} )
&  = & 
-\lp   \frac{c}{\sqrt{2}}\lp \y^{(a_1)} \rp^T\y^{(a_1)}  -  \frac{1}{\sqrt{2}c} \lp \x^{(a_1)} \rp^TS^TS\x^{(a_1)}     \rp^2
 \nonumber \\
&  = & 
-\lp   \frac{c}{\sqrt{2}}  -  \frac{1}{\sqrt{2}c} \|S\x^{(a_1)}\|_2^2     \rp^2
 \nonumber \\
&  = & 
-\lp   \frac{c}{\sqrt{2}}  -  \frac{c}{\sqrt{2}}     \rp^2
 \nonumber \\
&  = &   0,
  \end{eqnarray}
  where the second to last equality follows since $\|S\x^{(a_i)}\|_2=c$ and $\y^{(a_i)}\in\mS^m$ for $i=1,2$. Combining (\ref{eq:strmr1})--(\ref{eq:strmr5a0}) with correspondence  $Y\leftrightarrow\cG_u$ and $X\leftrightarrow\cG_1$ allows us to apply  Theorem \ref{thm:Gordonpos1}   to processes $\cG_1(\cdot)$ and  $\cG_u(\cdot)$ and obtain
\begin{align}\label{eq:strmt5a1a0}
&  &\mE \max_{\cX^{(a_1)}} \cG_1(\cX)  & \geq \mE \max_{\cX^{(a_1)}} \cG_u(\cX)
\nonumber \\
\Longleftrightarrow & & \mE \max_{\x\in\mS^d,\|S\x\|_2=c,\y\in\mS^n} \lp \sum_{i=1}^n \sum_{j=1}^m A_{i,j}\s_i\x_i\y_j  \rp & \geq \mE  \lp \frac{c}{\sqrt{2}} C(c)   +  \frac{1}{\sqrt{2} c} B(c) \rp
\nonumber \\
\Longleftrightarrow & & \mE \max_{\x\in\mS^d,\|S\x\|_2=c,\y\in\mS^n}  \y^TAS\x &  
\geq \mE   \lp \frac{c}{\sqrt{2}} C(c)   +  \frac{1}{\sqrt{2} c} B(c) \rp,
 \end{align}
where
\begin{eqnarray}
\label{eq:strmt5a1a0a0}
 C(c) = \max_{\y\in\mS^n} \y^TG^{(1)}\y, \quad\quad \mbox{and} \quad \quad  
  B(c) =\max_{\x\in\mS^d,\|S\x\|_2=c}  \x^TS^TG^{(2)}S\x.
 \end{eqnarray}
We then observe that $\max_{\y\in\mS^n}\y^TG^{(1)}\y$ corresponds to the ground state energy of the classical spherical SK model and is given by
\begin{eqnarray}
\label{eq:strmt5a1a0a1}
 \lim_{n\rightarrow \infty } \frac{1}{\sqrt{n} } \mE C(c) 
 =  \lim_{n\rightarrow \infty } \frac{1}{\sqrt{n} } \mE \max_{\y\in\mS^n} \y^TG^{(1)}\y =\sqrt{2} .
 \end{eqnarray}
Connecting  (\ref{eq:inteq1ab0}) and (\ref{eq:strmt5a1a0}), we also find
\begin{eqnarray}\label{eq:strmt5a1a1}
 \lim_{n\rightarrow \infty } \mE  \xi(c) =  \lim_{n\rightarrow \infty } \mE \frac{1}{\sqrt{n} }  \max_{\x\in\mS^d,\|S\x\|_2=c,\y\in\mS^n}  \y^TAS\x & \geq &  c +   \lim_{n\rightarrow \infty } \frac{1}{\sqrt{2n}c}  \mE B(c) .
  \end{eqnarray}
which, together with (\ref{eq:strthm1eq1}) and (\ref{eq:strmt5a1a0a0}), gives  (\ref{eq:strthm1eq2}) and completes the proof.
\end{proof}

\subsubsection{Handling $\mE B(c)$}
\label{sec:strhandbc}

We split handling $\mE B(c)$ into two parts. $\mE B(c)$'s upper-bound is discussed in the first part  and its matching lower bound in the second part.

\subsubsubsection{Upper-bounding $\mE B(c)$}
\label{sec:strhandbcub}

To upper-bound  $B(c)$ we closely follow procedures developed in previous sections. The following theorem summarizes main results.

\begin{theorem}
\label{thm:ubstrthm1}
Consider large $d\in\mN$ and let the elements of $G^{(2)}\in\mR^{d\times d}$, $\g^{(2)}\in\mR^{d\times 1}$, and $g\in\mR$ be independent standard normals ($G^{(2)}$, $\g^{(2)}$, and $g$ are independent of each other as well).   For $\s\in\mR_+^{d\times 1}$, diagonal $S\in\mR^{d\times d}$ such that $S=\mbox{diag}(\s)$, and $c\in(\min(\s),\max(\s))$, let $L(c)$ and $B(c)$ be as in (\ref{eq:thm1eq1}) and
 (\ref{eq:strthm1eq1}), respectively. One then has  
\begin{eqnarray}
   \label{eq:ubstrthm1eq2}
\frac{1}{\sqrt{2d}c} \mE B(c) \leq \frac{1}{\sqrt{d}} \mE L(c) .
\end{eqnarray}
\end{theorem}

\begin{proof} We consider two centered Gaussian processes indexed by an array $\cX = \{\x\}$
  \begin{eqnarray}
\label{eq:ubstrmr1}
 \cG_B (\cX) & \triangleq &  \cG_B (\x)  \triangleq  \sum_{i=1}^n \sum_{j=1}^m G^{(2)}_{i,j}\s_i\x_i\s_j\x_j +c^2 g   \nonumber   \\
 \cG_{B_u} (\cX) & \triangleq &  \cG_{B_u} (\x)  \triangleq      \sqrt{2} c \lp \g^{(2)}\rp^T S\x .
  \end{eqnarray}
We take two arrays $\cX^{(a_1)}=\{ \x^{(a_1)}\}$ and $\cX^{(a_2)}=\{ \x^{(a_2)}\}$ with $\|\x^{(a_i)}\|_2 =1$ and $\|S\x^{(a_i)}\|_2=c$, $i=1,2$ and write
  \begin{eqnarray}
\label{eq:ubstrmr2}
\mE \cG_B (\cX^{(a_1)})\cG_B (\cX^{(a_2)})   & =  &   \lp \lp \x^{(a_1)} \rp^TS^TS\x^{(a_2)} \rp^2  +c^4
\nonumber   \\
\mE \cG_{B_u} (\cX^{(a_1)})\cG_{B_u} (\cX^{(a_2)})  & =  &    2c^2   \lp \x^{(a_1)} \rp^TS^TS\x^{(a_2)} .
  \end{eqnarray}
Subtracting second from the first equation gives
  \begin{eqnarray}
\label{eq:ubstrmr5}
  \mE \cG_B (\cX^{(a_1)})\cG_B (\cX^{(a_2)})
 -
\mE \cG_{B_u} (\cX^{(a_1)})\cG_{B_u} (\cX^{(a_2)} )
 & = &    \lp \lp \x^{(a_1)} \rp^TS^TS\x^{(a_2)} \rp^2  +c^4
  -  2c^2   \lp \x^{(a_1)} \rp^TS^TS\x^{(a_2)} 
\nonumber
\\
& = &
 \lp  c^2 -  \lp \x^{(a_1)} \rp^TS^TS\x^{(a_2)}  \rp^2
  \geq   0. \nonumber \\
  \end{eqnarray}
Additionally, we also have  
  \begin{eqnarray}
\label{eq:ubstrmr5a0}
  \mE \cG_B (\cX^{(a_1)})\cG_B (\cX^{(a_1)})
 -
\mE \cG_{B_u} (\cX^{(a_1)})\cG_{B_u} (\cX^{(a_1)} )
&  = & 
 \lp  c^2 -  \lp \x^{(a_1)} \rp^TS^TS\x^{(a_2)}  \rp^2
 \nonumber \\
&  = & 
\lp    c^2  -   \|S\x^{(a_1)}\|_2^2     \rp^2
 \nonumber \\
&  = & 
\lp  c^2 -  c^2     \rp^2
 \nonumber \\
&  = &   0,
  \end{eqnarray}
  where the second to last equality follows since $\|S\x^{(a_i)}\|_2=c$ for $i=1,2$. Taking (\ref{eq:ubstrmr1})--(\ref{eq:ubstrmr5a0}) together with $Y\leftrightarrow\cG_B$ and $X\leftrightarrow\cG_{B_u}$ correspondence and applying  Theorem \ref{thm:Gordonpos1}   to processes $\cG_B(\cdot)$ and  $\cG_{B_u}(\cdot)$ gives
\begin{align}\label{eq:ubstrmt5a1a0}
&  &\mE \max_{\cX^{(a_1)}} \cG_B(\cX)  & \leq \mE \max_{\cX^{(a_1)}} \cG_{B_u}(\cX)
\nonumber \\
\Longleftrightarrow & & \mE \max_{\x\in\mS^d,\|S\x\|_2=c} \lp \sum_{i=1}^d  G_{i,j}^{(2)}\s_i\x_i\s_j\x_j  +c^2 g \rp & \leq \sqrt{2}c   \mE \max_{\x\in\mS^d,\|S\x\|_2=c}  \lp\g^{(2)}\rp^T\x 
\nonumber \\
\Longleftrightarrow & & \mE \max_{\x\in\mS^d,\|S\x\|_2=c} \x^TS^TG^{(2)}S\x & \leq \sqrt{2}c   \mE \max_{\x\in\mS^d,\|S\x\|_2=c}  \lp\g^{(2)}\rp^T\x 
\nonumber \\
\Longleftrightarrow & & \mE B(c)  & \leq \sqrt{2}c   \mE L(c)
\nonumber \\
\Longleftrightarrow & & \frac{1}{\sqrt{2d}c}\mE B(c)  & \leq \frac{1}{\sqrt{d}}   \mE L(c),
 \end{align}
which matches (\ref{eq:ubstrthm1eq2}) and completes the theorem's proof. 
\end{proof}

\subsubsubsection{Lower-bounding $\mE B(c)$}
\label{sec:strhandbclb}

As stated earlier, Theorem \ref{thm:Gordonpos1} is a special case of  concepts discussed in Corollary 3 in \cite{Stojnicgscompyx16}  and in Corollary 4 in \cite{Stojnicgscomp16}). The machinery developed there ensures that the upper-bounding mechanism of the previous subsection is tight provided that two nontrivially overlapped ($q\neq1$) replicated systems cannot double the maximal value of a single system. The very same principle was utilized in the single-partite system in \cite{TalSph06,TalSK06} which we consider here (albeit in a more general spin configuration).

Below, we check whether this condition indeed holds. Following \cite{Stojnicgscompyx16,Stojnicgscomp16,TalSph06}, for a real scalar $t\in [0,1]$, we first consider interpolated system
 \begin{eqnarray}
 \label{eq:lbstr1}
\mbox{ \textbf{1-rep:} } \quad\quad D(c;t)=  \max_{\x\in\mS^d,\|S\x\|_2=c}  \lp \sqrt{t}\x^TS^TG^{(2)}S\x +\sqrt{1-t}\sqrt{2}c \lp\g^{(2)}\rp^TS\x \rp.  
   \end{eqnarray}
Clearly, $D(c;t)$ continuously interpolates between $B(c)$ and $\sqrt{2}L(c)$. In particular,  one has
 \begin{eqnarray}
 \label{eq:lbstr2}
 D(c,1) = B(c) \quad\quad \mbox{ and } \quad\quad  
   D(c;0)= \sqrt{2}cL(c).
  \end{eqnarray}
Moreover, Theorem \ref{thm:ubstrthm1} gives
 \begin{eqnarray}
 \label{eq:lbstr2a0}
\frac{1}{\sqrt{2d}c}\mE B(c) = \frac{1}{\sqrt{2d}c} \mE D(c,1)  \leq \frac{1}{\sqrt{2d}c} \mE D(c,t) \leq  
\frac{1}{\sqrt{2d}c}  \mE D(c;0)= \frac{1}{\sqrt{d}} \mE L(c).
  \end{eqnarray}
Then consider a set of replica pairs
 \begin{eqnarray}
 \label{eq:lbstr3}
\bar{\cX} = \left \{ (\x^{(1)},\x^{(2)}) \hspace{.05in} | \hspace{.05in}
\x^{(1)},\x^{(2)}\in\mS^d,  \|S\x^{(1)}  \|_2= \|S\x^{(2)}  \|_2=c,\lp\x^{(1)}\rp^TS^TS\x^{(2)}=qc^2 \right  \} ,
  \end{eqnarray}
and associate to it the following $q$-overlapped replicated system
 \begin{eqnarray}
 \label{eq:lbstr3}
\mbox{ \textbf{2-$q$-rep:} } \quad\quad D^{(2)}(c;t)=  \max_{(\x^{(1)},\x^{(2)})\in\bar{\cX}}   \sum_{i=1}^{2} \lp \lp \x^{(i)}\rp^TS^TG^{(2)}S\x^{(i)} +\sqrt{1-t}\sqrt{2}c \lp\g^{(2)}\rp^TS\x^{(i)} \rp .  
   \end{eqnarray}
The above principle -- \emph{nontrivially overlapped replicated system cannot double the free energy} -- means that for any $t\in[0,1)$ and $q\in(-1,1)$
 \begin{eqnarray}
 \label{eq:lbstr4}
  \mbox{ \textbf{2-$q$-rep} } < \mbox{ $\mathbf{2}\times$(\textbf{1-rep}) } .  
   \end{eqnarray}
In mathematical terminology one then has for any $t\in[0,1)$
 \begin{eqnarray}
 \label{eq:lbstr5}
  \mbox{ \textbf{2-$q$-rep} } < \mbox{ $\mathbf{2}\times$(\textbf{1-rep}) } 
  \quad\quad
\Longleftrightarrow
  \quad\quad
  \min_{q\in(-1,1)} \lp 
  \lim_{d\rightarrow \infty} \frac{2}{\sqrt{d}} \mE D(c;0) 
  -
   \lim_{d\rightarrow \infty} \frac{1}{\sqrt{d}} \mE D^{(2)}(c;t) 
\rp >0.  
   \end{eqnarray}
If this principle is indeed in place, i.e., if one can show that (\ref{eq:lbstr5}) holds then complete analogues to Theorem 2.4 in \cite{TalSK06} and Theorem 5.2 in \cite{TalSph06} are established and  the remaining portions of the machineries of \cite{TalSph06,TalSK06} ensure $\mE D(c,1)=\mE D(c,0)$. In other words,
 \begin{equation}
 \label{eq:lbstr6}
 \min_{q\in(-1,1)} \lp 
  \lim_{d\rightarrow \infty} \frac{2}{\sqrt{2d}c} \mE D(c;0) 
  -
   \lim_{d\rightarrow \infty} \frac{1}{\sqrt{2d}c} \mE D^{(2)}(c;t) 
\rp >0 
\hspace{.04in}\Longleftrightarrow \hspace{.04in}
\lim_{d\rightarrow \infty} \frac{1}{\sqrt{2d}c} \mE D(c,1)= \lim_{d\rightarrow \infty} \frac{1}{\sqrt{2d}c} \mE D(c,0).  
   \end{equation}

To establish $  \mbox{ \textbf{2-$q$-rep} } < \mbox{ $\mathbf{2}\times$(\textbf{1-rep}) } $, we first prove the following theorem that upper-bounds $\mE D^{(2)}(c;t)$.

\begin{theorem}
\label{thm:lbstrthm1}
Consider large $d\in\mN$ and let the elements of $G^{(2)}\in\mR^{d\times d}$, $\g^{(1,1)},\g^{(1,2)},\g^{(2)}\in\mR^{d\times 1}$, and $g\in\mR$ be independent standard normals ($G^{(2)}$, $\g^{(1,1)}$, $\g^{(1,2)}$, $\g^{(2)}$, and $g$ are all independent among themselves as well).  Let $t\in(0,1)$ and $q\in(-1,1)$. For nonnegative $\s\in\mR^{d\times 1}$, diagonal $S\in\mR^{d\times d}$ such that $S=\mbox{diag}(\s)$, and $c\in(\min(\s),\max(\s))$, let $D^{(2)}(c;t)$   be as in (\ref{eq:lbstr3}). Set $\g^{(1,3)} = q\g^{(1,1)} + \sqrt{1-q^2}\g^{(1,2)}$,  
\begin{equation}
   \label{eq:lbstrthm1eq1}
\cG_{D_u} (\x^{(1)},\x^{(2)}) =    
\sqrt{2} c\lp \sqrt{t} \lp\g^{(1,1)}\rp^TS\x^{(1)}  
+\sqrt{t} \lp\g^{(1,3)}\rp^TS\x^{(2)}  
+\sqrt{1-t} \lp\g^{(2)}\rp^TS \lp\x^{(1)}+ \x^{(2)} \rp\rp,
\end{equation} 
and
\begin{equation}
   \label{eq:lbstrthm1eq1a0}
L^{(2)}(c;t)=  \max_{(\x^{(1)},\x^{(2)})\in\bar{\cX}}  
\frac{1}{\sqrt{2}c}\cG_{D_u} (\x^{(1)},\x^{(2)}) .
\end{equation} 
 One then has  
\begin{eqnarray}
   \label{eq:lbstrthm1eq2}
\frac{1}{\sqrt{2d}c} \mE D^{(2)}(c;t) \leq \frac{1}{\sqrt{d}} \mE L^{(2)}(c;t) .
\end{eqnarray}
\end{theorem}

\begin{proof} We consider two centered Gaussian processes indexed by an array $\cX = \{\x^{(1)},\x^{(2)}\}$
  \begin{eqnarray}
\label{eq:lbstrmr1}
 \cG_D (\cX)  \triangleq   \cG_D (\x^{(1)},\x^{(2)})  
 & \triangleq  & \sqrt{t}\sum_{i=1}^n \sum_{j=1}^m G^{(2)}_{i,j}\s_i\x_i^{(1)}\s_j\x_j^{(1)}+  \sqrt{t} \sum_{i=1}^n \sum_{j=1}^m G^{(2)}_{i,j}\s_i\x_i^{(2)}\s_j\x_j^{(2)} 
 \nonumber \\
 & & 
  +\sqrt{1-t} \sqrt{2}c\lp\g^{(2)}\rp^T\x^{(1)}+\sqrt{1-t}  \sqrt{2}c \lp\g^{(2)}\rp^T\x^{(2)} 
 + \sqrt{t}  \sqrt{2c^4(1+q^2)}  g  \nonumber   \\
 \cG_{D_u} (\cX) & \triangleq &  \cG_{D_u} (\x^{(1)},\x^{(2)})   .
  \end{eqnarray}
We take two arrays $\cX^{(a)}=\{ \x^{(a_1)}, \x^{(a_2)}\}$ and $\cX^{(b)}=\{ \x^{(b_1)}, \x^{(b_2)}\}$ with $\|\x^{(a_i)}\|_2 =\|\x^{(b_i)}\|_2 =1$,  $\|S\x^{(a_i)}\|_2=\|S\x^{(b_i)}\|_2=c$, $i=1,2$, and $ \lp\x^{(a_1)}\rp^TS^TS\x^{(a_2)}=\lp\x^{(b_1)}\rp^TS^TS\x^{(b_2)}=qc^2$ . Then we have 
  \begin{eqnarray}
\label{eq:lbstrmr2}
\mE \cG_D (\cX^{(a)})\cG_D (\cX^{(b)})   & =  &   
t \lp \lp \x^{(a_1)} \rp^TS^TS\x^{(b_1)} \rp^2 
+
t\lp \lp \x^{(a_1)} \rp^TS^TS\x^{(b_2)} \rp^2
+
t\lp \lp \x^{(a_2)} \rp^TS^TS\x^{(b_1)} \rp^2
\nonumber \\
& & 
 +
t\lp \lp \x^{(a_2)} \rp^TS^TS\x^{(b_2)} \rp^2
+
2c(1-t) \lp  \x^{(a_1)}  +  \x^{(a_2)} \rp^T\lp \x^{(b_1)} + \x^{(b_2)} \rp 
 \nonumber \\
& &  + (2c^4+2c^4q^2) t, 
  \end{eqnarray}
and
  \begin{eqnarray}
\label{eq:lbstrmr2a0}
 \mE \cG_{D_u} (\cX^{(a)})\cG_{D_u} (\cX^{(b)})  & =  &  
   2c^2t   \lp \x^{(a_1)} \rp^TS^TS\x^{(b_1)}
   +
   2c^2tq   \lp \x^{(a_1)} \rp^TS^TS\x^{(b_2)}
   +
  2c^2tq   \lp \x^{(a_2)} \rp^TS^TS\x^{(b_1)}
\nonumber \\
& &    +
      2c^2t   \lp \x^{(a_2)} \rp^TS^TS\x^{(b_2)} 
      +
2c(1-t) \lp  \x^{(a_1)}  +  \x^{(a_2)} \rp^T\lp \x^{(b_1)} + \x^{(b_2)} \rp .
  \end{eqnarray}
Subtracting (\ref{eq:lbstrmr2a0}) from (\ref{eq:lbstrmr2}) gives
  \begin{eqnarray}
\label{eq:lbstrmr5}
 & & \hspace{-.4in} \mE \cG_D (\cX^{(a)})\cG_D (\cX^{(b)})
  - 
 \mE \cG_{D_u} (\cX^{(a)})\cG_{D_u} (\cX^{(b)} )
  =      
\nonumber \\
& = & t \lp  c^2 -  \lp \x^{(a_1)} \rp^TS^TS\x^{(b1)}  \rp^2
+
t \lp  qc^2 -  \lp \x^{(a_1)} \rp^TS^TS\x^{(b2)}  \rp^2
\nonumber \\
& & 
+ t \lp  qc^2 -  \lp \x^{(a_2)} \rp^TS^TS\x^{(b1)}  \rp^2
+t \lp  c^2 -  \lp \x^{(a_2)} \rp^TS^TS\x^{(b2)}  \rp^2
  \geq   0.  
  \end{eqnarray}
We also have  
  \begin{eqnarray}
\label{eq:lbstrmr5a0}
 & & \hspace{-.4in} \mE \cG_D (\cX^{(a)})\cG_D (\cX^{(a)})
  - 
 \mE \cG_{D_u} (\cX^{(a)})\cG_{D_u} (\cX^{(a)} )
  =      
\nonumber \\
& = & t \lp  c^2 -  \lp \x^{(a_1)} \rp^TS^TS\x^{(a1)}  \rp^2
+
t \lp  qc^2 -  \lp \x^{(a_1)} \rp^TS^TS\x^{(a2)}  \rp^2
\nonumber \\
& & 
+ t \lp  qc^2 -  \lp \x^{(a_2)} \rp^TS^TS\x^{(a1)}  \rp^2
+t \lp  c^2 -  \lp \x^{(a_2)} \rp^TS^TS\x^{(a2)}  \rp^2
  \nonumber \\
&  = & 
\lp    c^2  -   \|S\x^{(a_1)}\|_2^2     \rp^2
+
t \lp  qc^2 -  \lp \x^{(a_1)} \rp^TS^TS\x^{(a2)}  \rp^2
\nonumber \\
& &  + t \lp  qc^2 -  \lp \x^{(a_2)} \rp^TS^TS\x^{(a1)}  \rp^2
+\lp    c^2  -   \|S\x^{(a_2)}\|_2^2     \rp^2
  \nonumber \\
&  = &   0,
  \end{eqnarray}
  where the last equality follows since $\|S\x^{(a_i)}\|_2=c$ for $i=1,2$ and $ \lp\x^{(a_1)}\rp^TS^TS\x^{(a_2)}=\lp\x^{(b_1)}\rp^TS^TS\x^{(b_2)}=qc^2$. Relying on (\ref{eq:lbstrmr1})--(\ref{eq:lbstrmr5a0}) and $Y\leftrightarrow\cG_D$ and $X\leftrightarrow\cG_{D_u}$ correspondence, we apply  Theorem \ref{thm:Gordonpos1}   to processes $\cG_D(\cdot)$ and  $\cG_{D_u}(\cdot)$ and obtain
\begin{align}\label{eq:lbstrmt5a1a0}
&  &\mE \max_{\cX^{(a_1,a_2)}} \cG_D(\cX)  & \leq \mE \max_{\cX^{(a_1,a_2)}} \cG_{D_u}(\cX)
\nonumber \\
\Longleftrightarrow & & \mE \max_{(\x^{(1)},\x^{(2)})\in\bar{\cX}} \cG_D (\x^{(1)},\x^{(2)}) & \leq \mE \max_{(\x^{(1)},\x^{(2)})\in\bar{\cX}} \cG_{D_u} (\x^{(1)},\x^{(2)}) 
 \nonumber \\
\Longleftrightarrow & & \mE D^{(2)}(c;t)  & \leq \sqrt{2}c   \mE L^{(2)}(c;t)
\nonumber \\
\Longleftrightarrow & & \frac{1}{\sqrt{2d}c}\mE D^{(2)}(c;t)  & \leq \frac{1}{\sqrt{d}}   \mE L^{(2)}(c;t),
 \end{align}
which matches (\ref{eq:lbstrthm1eq2}) and completes the theorem's proof. 
\end{proof}

Keeping in mind (\ref{eq:lbstr2}), (\ref{eq:lbstr2a0}), and (\ref{eq:lbstrthm1eq2}), condition
in (\ref{eq:lbstr5}) and  (\ref{eq:lbstr6}) will be met if for any $t<1$
\begin{eqnarray}
   \label{eq:lbstrmt5a1a1}
\lim_{d\rightarrow \infty} \frac{1}{\sqrt{d}} \mE L^{(2)}(c;t) < 2 \lim_{d\rightarrow \infty}\frac{1}{\sqrt{d}} \mE L(c).
\end{eqnarray}
We provide a characterization of $\lim_{d\rightarrow \infty} \frac{1}{\sqrt{d}} \mE L^{(2)}(c;t) $ in the following subsection.

\subsubsubsection{Characterization of $\lim_{d\rightarrow \infty} \frac{1}{\sqrt{d}} \mE L^{(2)}(c;t) $}
\label{sec:l2}

 Clearly, $L(c)$ is the key object of interest. Below we study it in detail. First, from 
 (\ref{eq:lbstrthm1eq1}) and  (\ref{eq:lbstrthm1eq1a0})
\begin{equation}
   \label{eq:l2lbstrthm1eq1}
L^{(2)}(c;t)=  \max_{(\x^{(1)},\x^{(2)})\in\bar{\cX}}  
 \lp \sqrt{t} \lp\g^{(1,1)}\rp^TS\x^{(1)}  
+\sqrt{t} \lp\g^{(1,3)}\rp^TS\x^{(2)}  
+\sqrt{1-t} \lp\g^{(2)}\rp^TS \lp\x^{(1)}+ \x^{(2)} \rp\rp.
\end{equation} 
 Let $\g^{(x,1)}\in\mR^{d\times 1}$ have independent standard normal components. Also, let $\g^{(x,2)}\in\mR^{d\times 1}$ have independent standard normal components. Additionally, let
\begin{equation}
   \label{eq:l2lbstrthm1eq1a0}
\mE \g_i^{(x,1)}\g_i^{(x,2)} = 1-t + qt \triangleq a, 1\leq i\leq d. 
 \end{equation} 
One can then replace  (\ref{eq:lbstrthm1eq1})  with its statistical equivalent 
\begin{equation}
   \label{eq:l2lbstrthm1eq1a1}
L^{(2)}(c;t)=  \max_{(\x^{(1)},\x^{(2)})\in\bar{\cX}}  
 \lp \lp\g^{(x,1)}\rp^TS\x^{(1)}  
 +
 \lp\g^{(x,2)}\rp^TS\x^{(2)} 
 \rp.
\end{equation} 
Basic algebraic transformations allow to rewrite (\ref{eq:l2lbstrthm1eq1a1}) as
\begin{eqnarray}
   \label{eq:l2hrd1}
L^{(2)}(c;t)= - \min_{(\x^{(1)},\x^{(2)})\in\bar{\cX}}  
 \lp \lp\g^{(x,1)}\rp^TS\x^{(1)}  
 +
 \lp\g^{(x,2)}\rp^TS\x^{(2)} 
 \rp.
\end{eqnarray}
We then have for the Lagrangian 
\begin{eqnarray}
   \label{eq:l2hrd2}
 \cL & = & \sum_{i=1}^{d} \lp \g_i^{(x,1)}\s_i\x_i^{(1)}  
 +
\g_i^{(x,2)}\s_i\x_i^{(2)} \rp 
  + \gamma_1 \sum_{i=1}^{d}\lp\x_i^{(1)}\rp^2 -\gamma_1
 +   \gamma_2 \sum_{i=1}^{d}\lp\x_i^{(2)}\rp^2 -\gamma_2
 \nonumber \\
 & &  
   + \gamma_{01} \sum_{i=1}^{d}\s_i^2\lp\x_i^{(2)}\rp^2 -\gamma_{01} c^2
      + \gamma_{02} \sum_{i=1}^{d}\s_i^2 \lp \x_i^{(2)}\rp^2 -\gamma_{02} c^2
      + \nu \sum_{i=1}^{d}\x_i^{(1)}\s_i^2\x_i^{(2)} -\nu q c^2.
\end{eqnarray}
A combination of  (\ref{eq:l2hrd1}) and  (\ref{eq:l2hrd2}) together with duality gives
\begin{eqnarray}
   \label{eq:l2hrd3}
 L^{(2)}(c;t)  = -\min_{\x^{(1)},\x^{(2)}}\max_{\gamma_1,\gamma_2,\gamma_{01},\gamma_{02},\nu} \cL
 \leq  -\max_{\gamma_1,\gamma_2,\gamma_{01},\gamma_{02},\nu} \min_{\x^{(1)},\x^{(2)}} \cL.
\end{eqnarray}
To optimize over $\x^{(1)}$ and $\x^{(2)}$, we first find derivatives
\begin{eqnarray}
   \label{eq:l2hrd4}
 \frac{d\cL}{d\x_i^{(1)}} & =  & \g_i^{(x,1)} \s_i + 2\gamma_1 \x_i^{(1)} + 2\gamma_{01} \s_i^2\x_i^{(1)}
 +\nu\s_i^2\x_i^{(2)}
 \nonumber \\
  \frac{d\cL}{d\x_i^{(2)}} & =  & \g_i^{(x,2)} \s_i + 2\gamma_2 \x_i^{(2)} + 2\gamma_{02} \s_i^2\x_i^{(2)}
 +\nu\s_i^2\x_i^{(1)}.
\end{eqnarray}
As $\gamma$'s are scalars and problem is completely symmetric in $\x^{(1)}$ and  $\x^{(2)}$, we have $\gamma_1=\gamma_2=\gamma$ and $\gamma_{01}=\gamma_{02}=\gamma_0$. Also, we set
\begin{eqnarray}
   \label{eq:l2hrd4a0}
\z_{1,i} & =  & \g_i^{(x,1)} \s_i  
 \nonumber \\
\z_{2.i} & =  & \g_i^{(x,2)} \s_i  
 \nonumber \\
 \v_i & = &  2(\gamma + \gamma_0 \s_i^2)
 \nonumber \\
 \nu_i & = &  \nu \s_i^2.
\end{eqnarray}
Keeping (\ref{eq:l2hrd4a0}) in mind and equalling  derivatives in (\ref{eq:l2hrd4}) to zero gives 
\begin{eqnarray}
   \label{eq:l2hrd4a1}
\z_{1,i} + \nu_i \x_i^{(2)} + \x_i^{(1)} \v_i & = & 0 \nonumber \\
\z_{2,i} + \nu_i \x_i^{(1)} + \x_i^{(2)} \v_i & = & 0 .
\end{eqnarray}
Summing and subtracting the above two equations gives
\begin{eqnarray}
   \label{eq:l2hrd4a2}
\x_i^{(1)}+\x_i^{(2)} & = &  -\frac{\z_{1,i}+\z_{2,i}}{\nu_i+\v_i}  
\nonumber \\
 \x_i^{(1)}-\x_i^{(2)} & = &   -\frac{-\z_{1,i}+\z_{2,i}}{\nu_i-\v_i},
\end{eqnarray}
and
\begin{eqnarray}
   \label{eq:l2hrd4a3}
\x_i^{(1)}  & = &  \frac{ \z_{1,i} \v_i - \z_{2,i}\nu_i}{\nu_i^2-\v_i^2} \nonumber \\
\x_i^{(2)} & = & \frac{ -\z_{1,i} \v_i + \z_{2,i}\nu_i}{\nu_i^2-\v_i^2} .
\end{eqnarray}
We can then rewrite (\ref{eq:l2hrd2}) as 
\begin{equation}
   \label{eq:l2hrd4a4}
 \cL  =  \sum_{i=1}^{d} \lp \z_{1,i} \x_i^{(1)}  
 +
\z_{2,i}\x_i^{(2)} \rp 
  + \frac{1}{2} \sum_{i=1}^{d}\lp\x_i^{(1)}\rp^2\v_i 
 +  \frac{1}{2}  \sum_{i=1}^{d}\lp\x_i^{(2)}\rp^2\v_i
   + \sum_{i=1}^{d}\nu_1\x_i^{(1)}\x_i^{(2)}
    -2\gamma
 -2\gamma_{0} c^2
        -\nu q c^2.
\end{equation}
Set
\begin{eqnarray}
   \label{eq:l2hrd4a5}
 K_{1,i} & = & 2\nu_i(-\z_{1,i}\nu_i + \z_{2,i}\v_i)(\z_{1,i}\v_i - \z_{2,i}\nu_i)
= 2 \nu_i(-z_{1,i}^2\nu_i\v_i + \z_{2,i}\z_{1,i}\v_i^2 + \z_{1,i}\z_{2,i}\nu_i^2-\z_{2,i}^2\nu_i\v_i)
\nonumber \\
 K_{2,i} & = & \v_i(-\z_{1,i}\nu_i + \z_{2,i}\v_i)^2 = \v_i(\z_{1,i}^2\nu_i^2 -2\z_{1,i}\z_{2,i}\nu_i\v_i + \z_{2,i}^2\v_i^2) 
\nonumber \\
K_{3,i} & = & \v_i(\z_{1,i}\v_i - \z_{2,i}\nu_i)^2 = \v_i(\z_{1,i}^2\v_i^2 -2\z_{1,i}\z_{2,i}\nu_i\v_i + \z_{2,i}^2\nu_i^2).
  \end{eqnarray}
Then
\begin{equation}
   \label{eq:l2hrd4a6}
 \cL  =  \sum_{i=1}^{d} \lp \z_{1,i} \x_i^{(1)}  
 +
\z_{2,i}\x_i^{(2)} \rp 
  + \frac{1}{2} \sum_{i=1}^{d} \frac{K_{1,i} + K_{2,i} + K_{3,i}  }{(\nu_i^2-\v_i^2)^2}    -2\gamma
 -2\gamma_{0} c^2
        -\nu q c^2.
\end{equation}
From (\ref{eq:l2hrd4a5}) we find
\begin{align}
   \label{eq:l2hrd4a7}
K_{1,i}+K_{2,i}+K_{3,i}
& = 
 ( -\z_{1,i}^2\nu_i^2\v_i + \v_i\z_{1,i}^2\v_i^2 ) +2(\nu_i^3\z_{1,i}\z_{2,i} -\z_{1,i}\z_{2,i}\nu_i\v_i^2 ) +(- 2\z_{2,i}^2\nu_i^2\v_i + \v_i\z_{2,i}^2\nu_i^2 +\v_i^3\z_{2,i}^2) 
\nonumber \\
& =  
\v_i\z_{1,i}^2( -\nu_i^2 + \v_i^2 ) +2\z_{1,i}\z_{2,i}\nu_i( \nu_i^2 - \v_i^2 ) + \z_{2,i}^2\v_i(-\nu_i^2 + \v_i^2)
\nonumber \\
& =
( -\v\z_{1,i}^2 + 2\z_{1,i}\z_{2,i}\nu_i -\z_{2,i}^2\v_i )(\nu_i^2-\v_i^2). 
\end{align}
Plugging this back in (\ref{eq:l2hrd4a6}) gives
\begin{equation}
   \label{eq:l2hrd4a8}
 \cL  =  \sum_{i=1}^{d} \lp \z_{1,i} \x_i^{(1)}  
 +
\z_{2,i}\x_i^{(2)} \rp 
  + \frac{1}{2} \sum_{i=1}^{d} \frac{  -\v_i\z_{1,i}^2 + 2\z_{1,i}\z_{2,i}\nu_i -\z_{2,i}^2\v_i }{\nu_i^2-\v_i^2}    -2\gamma
 -2\gamma_{0} c^2
        -\nu q c^2.
\end{equation}
We also observe
\begin{eqnarray}
   \label{eq:l2hrd4a9}
\z_{1,i}\x_i^{(1)} + \z_{2,i}\x_i^{(2)} & = & \z_{1,i}\frac{ \z_{1,i}\v_i - \z_{2,i}\nu_i}{\nu_i^2-\v_i^2} + \z_{2,i}\frac{-\z_{1,i}\nu_i + \z_{2,i}\v_i}{\nu_i^2-\v_i^2}
\nonumber \\
& = & \frac{ \z_{1,i}^2\v_i - 2\z_{1,i}\z_{2,i}\nu_i + \z_{2,i}^2\v_i}{\nu_i^2-\v_i^2} .
\end{eqnarray}
A combination of  (\ref{eq:l2hrd4a8}) and (\ref{eq:l2hrd4a9}) gives
\begin{equation}
   \label{eq:l2hrd4a10}
\min_{\x^{(1)},\x^{(2)}}  \cL  =  - \frac{1}{2} \sum_{i=1}^{d} \frac{  -\v_i\z_{1,i}^2 + 2\z_{1,i}\z_{2,i}\nu_i -\z_{2,i}^2\v_i }{\nu_i^2-\v_i^2}    -2\gamma
 -2\gamma_{0} c^2
        -\nu q c^2.
\end{equation}
After a change of variables 
\begin{eqnarray}
   \label{eq:l2hrd4a11}
 \nu & = &\nu_x\gamma_0
 \nonumber \\
 \nu_i & = &\nu_{x,i}\gamma_0
 \nonumber \\
 \nu_{x,i} & = &\nu_x\s_i^2
 \nonumber \\
 \gamma &=& \gamma_x\gamma_0
 \nonumber \\
 \v_{x,i} &=& 2(\gamma_x+\s_i^2)
 \nonumber \\
 \v_i &= & 2(\gamma +\gamma_0\s_i^2)= 2\gamma_0(\gamma_x+\s_i^2) = \v_{x,i}\gamma_0,
 \end{eqnarray}
(\ref{eq:l2hrd4a10}) can be rewritten as
\begin{equation}
   \label{eq:l2hrd4a12}
\min_{\x^{(1)},\x^{(2)}}  \cL  =  - \frac{1}{4\gamma_0} \sum_{i=1}^{d} 2\frac{  -\v_{x,i}\z_{1,i}^2 + 2\z_{1,i}\z_{2,i}\nu_{x,i} -\z_{2,i}^2\v_{x,i} }{\nu_{x,i}^2-\v_{x,i}^2}    - \gamma_0 ( 2\gamma_x
 +2 c^2
        +\nu_x q c^2).
\end{equation}
Maximization ove $\gamma_0$ gives
\begin{equation}
   \label{eq:l2hrd4a13}
\max_{\gamma_0}\min_{\x^{(1)},\x^{(2)}}  \cL  =  - \sqrt{ \lp \sum_{i=1}^{d} 2\frac{  -\v_{x,i}\z_{1,i}^2 + 2\z_{1,i}\z_{2,i}\nu_{x,i} -\z_{2,i}^2\v_{x,i} }{\nu_{x,i}^2-\v_{x,i}^2} \rp ( 2\gamma_x
 +2 c^2
        +\nu_x q c^2) }.
\end{equation}
Combining (\ref{eq:l2hrd3}), (\ref{eq:l2hrd4a11}), and (\ref{eq:l2hrd4a13}) and relying on the law of large numbers and concentrations, we obtain 
\begin{eqnarray}
   \label{eq:l2hrd14}
\lim_{d\rightarrow \infty }\frac{1}{\sqrt{d}} \mE L^{(2)}(c;t)  
& \leq & 
\lim_{d\rightarrow \infty } \min_{\gamma_x,\nu_x} 
\sqrt{ \lp \frac{1}{d} \sum_{i=1}^{d} 2\s_i^2\frac{  -2\v_{x,i} + 2a\nu_{x,i}  }{\nu_{x,i}^2-\v_{x,i}^2} \rp ( 2\gamma_x
 +2 c^2
        +\nu_x q c^2) }
        \nonumber \\
        & = & 
\lim_{d\rightarrow \infty }  \min_{\gamma_x,\nu_x} 
\sqrt{ \lp \frac{1}{d}  \sum_{i=1}^{d} 4\s_i^2\frac{  \v_{x,i} - a\nu_{x,i}  }{\v_{x,i}^2 - \nu_{x,i}^2} \rp ( 2\gamma_x
 +2 c^2
        +\nu_x q c^2) }
               \nonumber \\
        & = & 
\lim_{d\rightarrow \infty }  \min_{\gamma_x,\nu_x} 
\sqrt{ \lp \frac{1}{d}  \sum_{i=1}^{d} 4\s_i^2\frac{  2(\gamma_x+\s_i^2) - a\nu_x\s_i^2  }{4(\gamma_x+\s_i^2)^2 - \nu_x^2\s_i^4} \rp ( 2\gamma_x
 +2 c^2
        +\nu_x q c^2) }.
\end{eqnarray}

The following theorem summarizes the above characterization of $\lim_{d\rightarrow \infty }\frac{1}{\sqrt{d}} \mE L^{(2)}(c;t)  $.

\begin{theorem}
 \label{thm:lbstrthm2}  
Assume the setup of Theorem \ref{thm:lbstrthm1}. Let $\bar{L}$ be as in (\ref{eq:hrd9}). Set 
\begin{eqnarray}
   \label{eq:lbstrthm2eq1}
\bar{L}^{(2)} \triangleq
\lp \frac{1}{d}  \sum_{i=1}^{d} 4\s_i^2\frac{  2(\gamma_x+\s_i^2) - (1-t+qt)\nu_x\s_i^2  }{4(\gamma_x+\s_i^2)^2 - \nu_x^2\s_i^4} \rp ( 2\gamma_x
 +2 c^2
        +\nu_x q c^2) .
\end{eqnarray}
  One then has  
\begin{eqnarray}
   \label{eq:lbstrthm2eq2}
\lim_{d\rightarrow \infty } \frac{1}{\sqrt{2d}c} \mE D^{(2)}(c;t) 
 \leq  
\lim_{d\rightarrow \infty } \frac{1}{\sqrt{d}} \mE L^{(2)}(c;t) 
 \leq  
\lim_{d\rightarrow \infty } \sqrt{ \min_{\gamma_x,\nu_x}  \bar{L}^{(2)}}.
\end{eqnarray}
Moreover, provided 
\begin{eqnarray}
   \label{eq:lbstrthm2eq3}
\lim_{d\rightarrow \infty }  \sqrt{\max_{t\in (0,1),q\in(-1,1)}\min_{\gamma_x,\nu_x}  \bar{L}^{(2)} } 
 < 2\lim_{d\rightarrow \infty } \sqrt{ \min_{\gamma_x}  \bar{L}  } ,
\end{eqnarray}
one also has that the condition in (\ref{eq:lbstr5}) is met, which then implies 
\begin{eqnarray}
   \label{eq:lbstrthm2eq4} 
\lim_{d\rightarrow \infty } \frac{1}{\sqrt{2d}c} \mE D(c,1)= \lim_{d\rightarrow \infty } \frac{1}{\sqrt{2d}c} \mE D(c,0),
\end{eqnarray}
and based on (\ref{eq:lbstr2}) and (\ref{eq:lbstr2a0})
  \begin{eqnarray}
   \label{eq:lbstrthm2eq5} 
 \lim_{d\rightarrow \infty } \frac{1}{\sqrt{2d}c}\mE B(c)  = \lim_{d\rightarrow \infty } \frac{1}{\sqrt{d}} \mE L(c),
  \end{eqnarray}
which implies equality in (\ref{eq:thm3eq2}) as well.
\end{theorem} 

\begin{proof}
  The first part, i.e.,  (\ref{eq:lbstrthm2eq2}), follows from (\ref{eq:l2hrd14}) and the preceding discussion after one recalls the definition of $a$ from (\ref{eq:l2lbstrthm1eq1a0}). The second part and equations (\ref{eq:lbstrthm2eq3})-(\ref{eq:lbstrthm2eq5}) follow from  (\ref{eq:lbstr2}), (\ref{eq:lbstr2a0}), (\ref{eq:lbstr5}), and (\ref{eq:hrd8}). To see that (\ref{eq:lbstrthm2eq5}) indeed implies equality in (\ref{eq:thm3eq2}),
we first observe that  Theorems \ref{thm:thm1} and \ref{thm:strthm1} (equations (\ref{eq:thm1eq2}) and (\ref{eq:strthm1eq2})) give
 \begin{eqnarray}
   \label{eq:adellam2}
c +  \frac{1}{\sqrt{2n}c} \mE B(c) \leq \mE \xi(c)  \leq c +  \frac{1}{\sqrt{n}} \mE L(c).
\end{eqnarray}
From (\ref{eq:lbstrthm2eq5}) and (\ref{eq:adellam2}) we then obatin
 \begin{eqnarray}
   \label{eq:adellam4}
\lim_{d\rightarrow \infty } \mE \xi(c)  = c +  \lim_{d\rightarrow \infty } \frac{1}{\sqrt{n}} \mE L(c).
= c +  \frac{1}{\sqrt{\alpha}}\lim_{d\rightarrow \infty } \frac{1}{\sqrt{d}} \mE L(c).
\end{eqnarray}
A combination of (\ref{eq:hrd8})--(\ref{eq:hrd11}) and results of Theorem \ref{thm:thm3} then gives that one indeed has equality in (\ref{eq:thm3eq2}). 
\end{proof}

\begin{remark}
\label{rem:rem2}
Throughout the entire analysis, we considered deterministic nature of the true covariance. In that case $\s$ is taken as a converging empirical distribution. The problem can be much easier if the true covariance is random with a specified pdf. Empirical averages should be then replaced by the analytical ones obtained for a given spectral pdf. In other words, all formulas given in this Section \ref{sec:handlerd} remain valid if summations are replaced by integrations over given pdfs.
\end{remark}

\subsubsubsection{Checking condition (\ref{eq:lbstrthm2eq3})}
\label{sec:chkcond}

Theorem  \ref{thm:lbstrthm2} provides all needed ingredients so that one can numerically check whether (\ref{eq:lbstrthm2eq3}) holds. As noted in  Remark \ref{rem:rem2}, when  the true covariance $\Sigma$ is random and characterized via its spectral density, all empirical averages become integrals. One then numerically evaluates $\bar{L}^{(2)}$ and  $\bar{L}$ and if (\ref{eq:lbstrthm2eq3}) holds, the proof is completed up to the level of numerical precision. Since numerical precision can be arbitrary one can effectively approach equality in (\ref{eq:lbstrthm2eq5}) arbitrarily closely. We tested quite a few ensembles and always obtained that (\ref{eq:lbstrthm2eq3}) holds.

It is even more interesting to apply the same methodology to deterministic setup presented in the previous section. In that case one can not numerically rely on infinite sizes. Nonetheless, one can expect the very same behavior for sufficiently large $d$. To see whether this indeed happens, we conducted a set of numerical evaluations and the obtained results are shown in Figure \ref{fig:fig1}. We selected $d=3000$ and $\s=\mbox{linspace}_d[0.5,1]$ ($\mbox{linspace}_d[0.5,1]$ is an increasing  vector of $d$ components equally spaced between  $0.5$ and $1$; in other words, $\s_i-\s_{i-1}=\s_j-\s_{j-1}$ for any $i$ and $j$; both endpoints of the interval are included as well, i.e., $\s_1=0.5$ and $\s_d=1$). While in theory $d\rightarrow \infty$, we found that $d$ on the order of a few thousands already behaves very well. As Figure \ref{fig:fig1} indicates, there are no local optima and the condition (\ref{eq:lbstrthm2eq3})  from Theorem  \ref{thm:lbstrthm2}  is indeed
 met (none of the curves reaches the maximum except at $q=1$).
 
   \begin{figure}[h]
\centering
\centerline{\includegraphics[width=.87\linewidth]{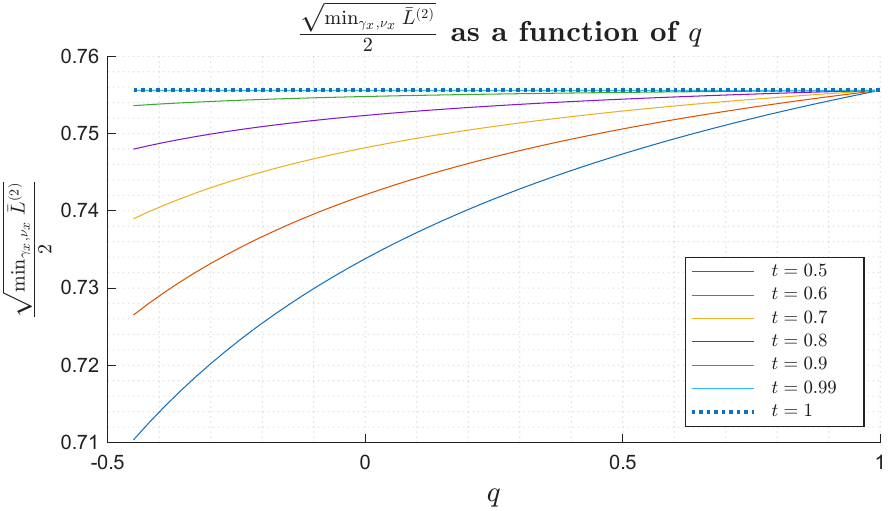}}
\caption{ $\frac{\sqrt{\min_{\gamma_x,\nu_x}\bar{L}^{(2)}}}{2}$ as a function of $q$ for varying $t$; $\s=\mbox{linspace}[0.5,1]$; $d=3000$ }
\label{fig:fig1}
\end{figure}

We complement the above numerical findings with the following theorem which analytically confirms that condition (\ref{eq:lbstrthm2eq3}) is met.
\begin{theorem}
 \label{thm:lbstrthm3}  
Assume the setup of Theorems \ref{thm:lbstrthm1} and \ref{thm:lbstrthm2}. One then has that (\ref{eq:lbstrthm2eq3}) holds.
\end{theorem}

\begin{proof}
Assume the opposite, i.e., that (\ref{eq:lbstrthm2eq3}) does not hold. Keeping in mind (\ref{eq:hrd9})  and (\ref{eq:lbstrthm2eq1}), this means that there are some $\hat{\gamma}_x$, $\hat{\nu}_x$, and $\tilde{\gamma}_x$, that satisfy 
\begin{eqnarray}
   \label{eq:lemprf0}
(\hat{\gamma}_x,\hat{\nu}_x) = \mbox{arg}\min_{\gamma_x,\nu_x}\bar{L}^{(2)} \quad\quad \mbox{and}\quad\quad
\tilde{\gamma}_x   = \mbox{arg}\min_{\gamma_x}\bar{L},
\end{eqnarray}
and for which the following holds
\begin{eqnarray}
   \label{eq:lemprf1}
\lp \frac{1}{d}  \sum_{i=1}^{d} 4\s_i^2\frac{  2(\hat{\gamma}_x+\s_i^2) - (1-t+qt)\hat{\nu}_x\s_i^2  }{4(\hat{\gamma}_x+\s_i^2)^2 - \hat{\nu}_x^2\s_i^4} \rp ( 2\hat{\gamma}_x
 +2 c^2
        +\hat{\nu}_x q c^2) 
        & = & 
        \min_{\gamma_x,\nu_x}\bar{L}^{(2)} 
        \nonumber \\
& = & 4\min_{\gamma_x}\bar{L} 
\nonumber \\
& = & 
4  \lp \frac{1}{d}\sum_{i=1}^{d}\frac{  \s_i^2}{\tilde{\gamma}_x  + \s_i^2} 
 (\tilde{\gamma}_x  + c^2) \rp  .
\end{eqnarray}
Moreover,
\begin{eqnarray}
   \label{eq:lemprf2}
 \left . \frac{d\bar{L}^{(2)}}{d\gamma_x} \right |_{(\gamma_x,\nu_x)=( \hat{\gamma}_x,\hat{\nu}_x)}
 =
 \left . \frac{d\bar{L}^{(2)}}{d\nu_x} \right |_{(\gamma_x,\nu_x)=( \hat{\gamma}_x,\hat{\nu}_x)}
 =\left . \frac{d\bar{L}}{d\gamma_x} \right |_{\gamma_x= \tilde{\gamma}_x}= 0 .
 \end{eqnarray}
We then observe that the choice $(\hat{\gamma}_x,\hat{\nu}_x)=(\tilde{\gamma}_x,0)$ satisfies (\ref{eq:lemprf1}) since
\begin{eqnarray}
   \label{eq:lemprf3}
\lp \frac{1}{d}  \sum_{i=1}^{d} 4\s_i^2\frac{  2(\tilde{\gamma}_x+\s_i^2) - (1-t+qt)0\s_i^2  }{4(\tilde{\gamma}_x+\s_i^2)^2 - 0^2\s_i^4} \rp ( 2 \tilde{\gamma}_x
 +2 c^2
        + 0 q c^2) 
 & = & 
 \lp \frac{1}{d}  \sum_{i=1}^{d} 4\s_i^2\frac{  2( \tilde{\gamma}_x+\s_i^2)    }{4(\tilde{\gamma}_x+\s_i^2)^2  } \rp ( 2\tilde{\gamma}_x
 +2 c^2 ) 
\nonumber \\& = & 
 \lp \frac{1}{d}  \sum_{i=1}^{d} \frac{ 2\s_i^2   }{(\tilde{\gamma}_x+\s_i^2) } \rp ( 2\tilde{\gamma}_x
 +2 c^2 ) 
\nonumber \\
& = & 
 4 \lp \frac{1}{d}\sum_{i=1}^{d}\frac{  \s_i^2}{ \tilde{\gamma}_x  + \s_i^2} 
 (\tilde{\gamma}_x  + c^2) \rp  .
\end{eqnarray}
Utilization of (\ref{eq:hrd12}) gives
 \begin{eqnarray}
   \label{eq:lemprf4}
\left . \frac{d\bar{L}}{d\gamma_x} \right |_{\gamma_x= \tilde{\gamma}_x }  =
   \frac{1}{d}\sum_{i=1}^{d}\frac{ \s_i^2( \s_i^2 -c^2)}{( \tilde{\gamma}_x  + \s_i^2)^2} = 0.
\end{eqnarray}
 We further find
 \begin{eqnarray}
   \label{eq:lemprf5}
  \frac{d\bar{L}^{(2)}}{d\gamma_x} 
  & = &
\frac{d}{d\gamma_x}  \lp \frac{1}{d}  \sum_{i=1}^{d} 4\s_i^2\frac{  2(\gamma_x+\s_i^2) - (1-t+qt)\nu_x\s_i^2  }{4(\gamma_x+\s_i^2)^2 - \nu_x^2\s_i^4} \rp ( 2\gamma_x
 +2 c^2
        +\nu_x q c^2) 
\nonumber \\
& & +
2\lp \frac{1}{d}  \sum_{i=1}^{d} 4\s_i^2\frac{  2( \gamma_x+\s_i^2) - (1-t+qt) \nu_x\s_i^2  }{4( \gamma_x+\s_i^2)^2 - \nu_x^2\s_i^4} \rp 
\nonumber \\
   & = &
  \lp \frac{1}{d}  \sum_{i=1}^{d} 4\s_i^2\frac{  2  }{4(\gamma_x+\s_i^2)^2 - \nu_x^2\s_i^4} \rp ( 2\gamma_x
 +2 c^2
        +\nu_x q c^2) 
\nonumber \\
& & -
  \lp \frac{1}{d}  \sum_{i=1}^{d} 4\s_i^2\frac{  2(\gamma_x+\s_i^2) - (1-t+qt)\nu_x\s_i^2  }{ ( 4(\gamma_x+\s_i^2)^2 - \nu_x^2\s_i^4 )^2 }8(\gamma_x+\s_i^2) \rp ( 2\gamma_x
 +2 c^2
        +\nu_x q c^2) 
\nonumber \\
& & +
2\lp \frac{1}{d}  \sum_{i=1}^{d} 4\s_i^2\frac{  2( \gamma_x+\s_i^2) - (1-t+qt) \nu_x\s_i^2  }{4( \gamma_x+\s_i^2)^2 - \nu_x^2\s_i^4} \rp .
 \end{eqnarray}
Evaluating for $(\gamma_x,\nu_x)=(\tilde{\gamma}_x,0)$ gives
 \begin{eqnarray}
   \label{eq:lemprf6}
\left .  \frac{d\bar{L}^{(2)}}{d\gamma_x} \right |_{(\gamma_x,\nu_x)=(\tilde{\gamma}_x,0)} 
    & = &
\  \lp \frac{1}{d}  \sum_{i=1}^{d} 4\s_i^2\frac{  2  }{4(\tilde{\gamma}_x+\s_i^2)^2 } \rp ( 2 \tilde{\gamma}_x
 +2 c^2) 
\nonumber \\
& &  -
  \lp \frac{1}{d}  \sum_{i=1}^{d} 4\s_i^2\frac{  2(\tilde{\gamma}_x+\s_i^2)   }{ ( 4(\tilde{\gamma}_x+\s_i^2)^2 )^2 }8(\tilde{\gamma}_x+\s_i^2) \rp ( 2 \tilde{\gamma}_x
 +2 c^2) 
  +
2\lp \frac{1}{d}  \sum_{i=1}^{d} 4\s_i^2\frac{  2( \tilde{\gamma}_x +\s_i^2)    }{4( \tilde{\gamma}_x+\s_i^2)^2  } \rp 
\nonumber \\
& =& 
4\lp \frac{1}{d}  \sum_{i=1}^{d} \s_i^2\frac{  ( \s_i^2 -c^2 )    }{( \tilde{\gamma}_x +\s_i^2)^2  } \rp 
\nonumber \\
& = & 
0,
 \end{eqnarray}
 where the last equality follows from (\ref{eq:lemprf4}). For the $\nu$ derivative, we find
 \begin{eqnarray}
   \label{eq:lemprf7}
  \frac{d\bar{L}^{(2)}}{d\nu_x} 
  & = &
\frac{d}{d\nu_x}  \lp \frac{1}{d}  \sum_{i=1}^{d} 4\s_i^2\frac{  2(\gamma_x+\s_i^2) - (1-t+qt)\nu_x\s_i^2  }{4(\gamma_x+\s_i^2)^2 - \nu_x^2\s_i^4} \rp ( 2\gamma_x
 +2 c^2
        +\nu_x q c^2) 
\nonumber \\
& & +
qc^2\lp \frac{1}{d}  \sum_{i=1}^{d} 4\s_i^2\frac{  2( \gamma_x+\s_i^2) - (1-t+qt) \nu_x\s_i^2  }{4( \gamma_x+\s_i^2)^2 - \nu_x^2\s_i^4} \rp 
\nonumber \\
   & = &
  -\lp \frac{1}{d}  \sum_{i=1}^{d} 4\s_i^2\frac{  (1-t+qt)\s_i^2  }{4(\gamma_x+\s_i^2)^2 - \nu_x^2\s_i^4} \rp ( 2\gamma_x
 +2 c^2
        +\nu_x q c^2) 
\nonumber \\
& & +
  \lp \frac{1}{d}  \sum_{i=1}^{d} 4\s_i^2\frac{  2(\gamma_x+\s_i^2) - (1-t+qt)\nu_x\s_i^2  }{ ( 4(\gamma_x+\s_i^2)^2 - \nu_x^2\s_i^4 )^2 }2(\nu_x\s_i^4) \rp ( 2\gamma_x
 +2 c^2
        +\nu_x q c^2) 
\nonumber \\
& & +
qc^2\lp \frac{1}{d}  \sum_{i=1}^{d} 4\s_i^2\frac{  2( \gamma_x+\s_i^2) - (1-t+qt) \nu_x\s_i^2  }{4( \gamma_x+\s_i^2)^2 - \nu_x^2\s_i^4} \rp .
 \end{eqnarray}
Evaluating again for $(\gamma_x,\nu_x)=(\tilde{\gamma}_x,0)$ gives
 \begin{eqnarray}
   \label{eq:lemprf8}
\left .  \frac{d\bar{L}^{(2)}}{d\nu_x} \right |_{(\gamma_x,\nu_x)=(\tilde{\gamma}_x,0)}
    & = &
  -\lp \frac{1}{d}  \sum_{i=1}^{d} 4\s_i^2\frac{  (1-t+qt)\s_i^2  }{4(\tilde{\gamma}_x+\s_i^2)^2 } \rp ( 2\tilde{\gamma}_x
 +2 c^2) 
\nonumber \\
 & & +
qc^2\lp \frac{1}{d}  \sum_{i=1}^{d} 4\s_i^2\frac{  2( \tilde{\gamma}_x+\s_i^2)  }{4(\tilde{\gamma}_x+\s_i^2)^2  } \rp 
    \nonumber \\
    & = &
  -\lp \frac{1}{d}  \sum_{i=1}^{d} 4\s_i^2\frac{  (1-t+qt)\s_i^2  }{4(\tilde{\gamma}_x+\s_i^2)^2 } \rp ( 2\tilde{\gamma}_x
 +2 c^2) 
 +
qc^2\lp \frac{1}{d}  \sum_{i=1}^{d} 4\s_i^2\frac{  2( \tilde{\gamma}_x +c^2)  }{4( \tilde{\gamma}_x+\s_i^2)^2  } \rp 
\nonumber \\
 & & +
qc^2\lp \frac{1}{d}  \sum_{i=1}^{d} 4\s_i^2\frac{  2( \tilde{\gamma}_x + \s_i^2)  }{4( \tilde{\gamma}_x + \s_i^2)^2  } \rp 
-
qc^2\lp \frac{1}{d}  \sum_{i=1}^{d} 4\s_i^2\frac{  2( \tilde{\gamma}_x + c^2)  }{4( \tilde{\gamma}_x +\s_i^2)^2  } \rp 
\nonumber \\
 & = &
  \lp \frac{1}{d}  \sum_{i=1}^{d} 4\s_i^2\frac{ qc^2 - (1-t+qt)\s_i^2  }{4(\tilde{\gamma}_x+\s_i^2)^2 } \rp ( 2\tilde{\gamma}_x
 +2 c^2) 
+qc^2\lp \frac{1}{d}  \sum_{i=1}^{d} 4\s_i^2\frac{  2( \s_i^2-c^2 )  }{4( \tilde{\gamma}_x+\s_i^2)^2  } \rp 
\nonumber \\
 & = &
  \lp \frac{1}{d}  \sum_{i=1}^{d} 4\s_i^2\frac{ qc^2 - (1-t+qt)\s_i^2  }{4(\tilde{\gamma}_x+\s_i^2)^2 } \rp ( 2\tilde{\gamma}_x +2 c^2) 
\nonumber \\
 & = &
  \lp \frac{1}{d}  \sum_{i=1}^{d} \s_i^2\frac{ qc^2 -q\s_i^2 +q\s_i^2 - (1-t+qt)\s_i^2  }{(\tilde{\gamma}_x+\s_i^2)^2 } \rp ( 2\tilde{\gamma}_x +2 c^2) 
\nonumber \\
 & = &
  \lp \frac{1}{d}  \sum_{i=1}^{d} \s_i^2\frac{ qc^2 -q\s_i^2   }{(\tilde{\gamma}_x+\s_i^2)^2 } \rp ( 2\tilde{\gamma}_x +2 c^2) 
  \nonumber \\
 &  &
  + \lp \frac{1}{d}  \sum_{i=1}^{d} \s_i^2\frac{  q\s_i^2 - (1-t+qt)\s_i^2  }{(\tilde{\gamma}_x+\s_i^2)^2 } \rp ( 2\tilde{\gamma}_x +2 c^2) 
\nonumber \\
 & = &
    \lp \frac{1}{d}  \sum_{i=1}^{d} \s_i^2\frac{  q\s_i^2 - (1-t+qt)\s_i^2  }{(\tilde{\gamma}_x + \s_i^2)^2 } \rp ( 2 \tilde{\gamma}_x +2 c^2) 
\nonumber \\
 & = &
 -   \lp \frac{1}{d}  \sum_{i=1}^{d} \s_i^4\frac{ (1-t)(1-q)  }{(\tilde{\gamma}_x +\s_i^2)^2 } \rp ( 2\tilde{\gamma}_x +2 c^2) 
\nonumber \\
 & \neq &
0,   \end{eqnarray}
 where fourth and seventh equality follow from (\ref{eq:lemprf4}). The last expression is different from zero  since $t\in(0,1)$ and $q\in(-1,1)$ (from (\ref{eq:hrd15a0}), one also has $\tilde{\gamma}_x\leq 2c^2$ which implies that $\tilde{\gamma}_x\neq c^2$ for optimal $c$). From (\ref{eq:lemprf8}), one has that $(\tilde{\gamma}_x,0)$ cannot be an $\bar{L}^{(2)} $'s stationary point, which contradicts (\ref{eq:lemprf2}) and completes the proof.
\end{proof}

\subsection{Matching $\delta$ and  $\mE\lambda_n(\hat{\Sigma}-\Sigma)$}
\label{sec:matchdellam}

We first recall/summarize all key aspects of the above analysis. From (\ref{eq:inteq1ab00}), we have
\begin{eqnarray}\label{eq:dellam1}
\lambda_n\lp \hat{\Sigma} -\Sigma   \rp 
   &  =  &
 \max_{c} \lp \lp \xi(c)\rp^2 - c^2 \rp.
\end{eqnarray}
Theorems \ref{thm:thm1} and \ref{thm:strthm1} (equations (\ref{eq:thm1eq2}) and (\ref{eq:strthm1eq2})) give
 \begin{eqnarray}
   \label{eq:dellam2}
c +  \frac{1}{\sqrt{2n}c} \mE B(c) \leq \mE \xi(c)  \leq c +  \frac{1}{\sqrt{n}} \mE L(c).
\end{eqnarray}
Also, Theorems \ref{thm:lbstrthm2} and \ref{thm:lbstrthm3}  give
   \begin{eqnarray}
   \label{eq:dellam3}
 \lim_{d\rightarrow \infty } \frac{1}{\sqrt{2d}c}\mE B(c)  = \lim_{d\rightarrow \infty } \frac{1}{\sqrt{d}} \mE L(c), 
  \end{eqnarray}
 and
  \begin{eqnarray}
   \label{eq:dellam4}
\lim_{d\rightarrow \infty } \mE \xi(c)   = c +  \frac{1}{\sqrt{\alpha}}\lim_{d\rightarrow \infty } \frac{1}{\sqrt{d}} \mE L(c).
\end{eqnarray}
Combining (\ref{eq:dellam1}) and (\ref{eq:dellam4}), we further find
\begin{eqnarray}\label{eq:dellam5}
\lim_{d\rightarrow \infty } \mE  \lambda_n\lp \hat{\Sigma} -\Sigma   \rp 
   &  =  & \max_c \lp  \frac{2c}{\sqrt{\alpha}}\lim_{d\rightarrow \infty } \frac{1}{\sqrt{d}} \mE L(c)
+
\frac{1}{\alpha}\lp \lim_{d\rightarrow \infty } \frac{1}{\sqrt{d}} \mE L(c) \rp^2 \rp.
 \end{eqnarray}
 The above is already a very convenient characterization of sample covariance estimation error. It basically relates to the maximal eigenvalue of a residual error matrix $\hat{\Sigma}-\Sigma$. One can go a step further and show that this also matches the corresponding characterization of the spectral norm as well. 

In particular, from (\ref{eq:amat1a0a6}), we have
\begin{equation}\label{eq:delam7}
 \|\hat{\Sigma} - \Sigma\|_2 = \max(\lambda_n(\hat{\Sigma} - \Sigma),|\lambda_1(\hat{\Sigma} - \Sigma)|). 
\end{equation}
The discussion from previous sections ensures that the righthand side of (\ref{eq:dellam4}), i.e., $\lim_{d\rightarrow \infty } \mE  \lambda_n ( \hat{\Sigma} -\Sigma ) $, is positive. If $\lim_{d\rightarrow \infty } \mE  \lambda_1(\hat{\Sigma} - \Sigma)\geq 0$ one then automatically has
\begin{equation}\label{eq:delam9}
\lim_{d\rightarrow \infty } \mE  \|\hat{\Sigma} - \Sigma\|_2 = \lim_{d\rightarrow \infty } \mE \lambda_n(\hat{\Sigma} - \Sigma) . 
\end{equation}
Therefore, the only interesting remaining scenario to consider is $\lim_{d\rightarrow \infty } \mE  \lambda_1 (\hat{\Sigma} - \Sigma) < 0$. To that end, we observe that when $\lambda_1 (\hat{\Sigma} - \Sigma) < 0$ then
\begin{eqnarray}\label{eq:dellam10}
\left |\lambda_1 \lp\hat{\Sigma} - \Sigma\rp \right | = \lambda_n \lp\Sigma -\hat{\Sigma}\rp
\end{eqnarray}
Following (\ref{eq:inteq1ad0})--(\ref{eq:inteq1ad2})
 \begin{eqnarray}\label{eq:dellam11}
\lambda_n \lp\Sigma -\hat{\Sigma}\rp  &  =  &
 \max_{\x\in\mS^d} \lp \x^T SS^T \x  -\frac{1}{n}\x^T S^TA^TAS \x \rp.
\end{eqnarray}
After setting
\begin{eqnarray}\label{eq:dellam12}
 \xi^-(c) & = &   
 \frac{1}{\sqrt{n}} \min_{\x\in\mS^d,\|S\x\|_2=c }  \sqrt{\x^T S^TA^TAS \x}
 =
 \frac{1}{\sqrt{n}} \min_{\x\in\mS^d,\|S\x\|_2=c }\max_{\y\in\mS^n }  \y AS \x ,
\end{eqnarray}
we have in place of (\ref{eq:dellam11}) 
\begin{eqnarray}\label{eq:dellam13}
\lambda_n \lp\Sigma -\hat{\Sigma}\rp  &  =  &
 \max_{\x\in\mS^d} \lp  \x^T SS^T \x  - \frac{1}{n}\x^T S^TA^TAS \x \rp 
 \nonumber \\
  &  =  &
 \max_{\x\in\mS^d,\|S\x\|_2=c,c} \lp \x^T SS^T \x - \frac{1}{n}\x^T S^TA^TAS \x  \rp 
 \nonumber \\
  &  =  &
 \max_{c} \lp  c^2 -  \min_{\x\in\mS^d,\|S\x\|_2=c} \frac{1}{n}\x^T S^TA^TAS \x\rp 
 \nonumber \\
  &  =  &
 \max_{c} \lp  c^2 - \lp \xi^-(c)\rp^2 \rp.
\end{eqnarray}
To characterize $\xi^-(c)$ and consequently $\lambda_n\lp \Sigma  - \hat{\Sigma}  \rp $, we again utilize RDT. The following theorem summarizes the obtained results. 

\begin{theorem}
\label{thm:dellamthm1}
Assume the setup of Theorem \ref{thm:thm1}. Let $\xi^-(c)$ be as in (\ref{eq:dellam12}).  One then has  
\begin{eqnarray}
   \label{eq:dellamthm1eq2}
\mE \xi^-(c)  \geq  \max \left \{c -  \frac{1}{\sqrt{n}} \mE L(c),0 \right \}.
\end{eqnarray}
\end{theorem}

\begin{proof} let processes $ \cG (\cX) $ and $ \cG_l (\cX) $ analogously to (\ref{eq:mr1}), i.e., let 
  \begin{eqnarray}
\label{eq:dellammr1}
 \cG (\cX) & \triangleq &  \cG (\x,\y)  \triangleq  \sum_{i=1}^n \sum_{j=1}^m A_{i,j}\s_i\x_i\y_j  + c g  \nonumber   \\
 \cG_l (\cX) & \triangleq &  \cG_l (\x,\y)  \triangleq    c\lp\g^{(1)}\rp^T\y +   \lp\g^{(2)}\rp^TS\x .
  \end{eqnarray}
For two arrays $\cX^{(a_1,b_1)}=\{ \x^{(a_1)},\y^{(b_1)}\}$ and $\cX^{(a_2,b_2)}=\{ \x^{(a_2)},\y^{(b_2)}\}$ with $\|\x^{(a_i)}\|_2=\|\y^{(b_i)}\|_2=1$ and $\|S\x^{(a_i)}\|_2=c$, $i=1,2$,  we have similarly to (\ref{eq:mr2})
  \begin{eqnarray}
\label{eq:dellammr2}
\mE \cG (\cX^{(a_1,b_1)})\cG (\cX^{(a_2,b_2)})   & =  &    \lp \x^{(a_1)} \rp^TS^TS\x^{(a_2)} \lp \y^{(b_1)} \rp^T\y^{(b_2)}  + c^2
\nonumber   \\
\mE \cG_l (\cX^{(a_1,b_1)})\cG_l (\cX^{(a_2,b_2)})  & =  &  c^2\lp \y^{(b_1)} \rp^T\y^{(b_2)}  +   \lp \x^{(a_1)} \rp^TS^TS\x^{(a_2)} .
  \end{eqnarray}
  From (\ref{eq:mr5}), we also find
  \begin{eqnarray}
\label{eq:dellammr5}
  \mE \cG (\cX^{(a_1,b_1)})\cG (\cX^{(a_2,b_2)})
 -
\mE \cG_l (\cX^{(a_1,b_1)})\cG_l (\cX^{(a_2,b_2)} )
 & = &    \lp \x^{(a_1)} \rp^TS^TS\x^{(a_2)} \lp \y^{(b_1)} \rp^T\y^{(b_2)}  + c^2 
\nonumber
\\
& &  - c^2\lp \y^{(b_1)} \rp^T\y^{(b_2)}  -   \lp \x^{(a_1)} \rp^TS^TS\x^{(a_2)}
\nonumber
\\
& = &
 \lp c^2- \lp \x^{(a_1)} \rp^TS^TS\x^{(a_2)}  \rp \lp 1 - \lp \y^{(b_1)} \rp^T\y^{(b_2)}\rp
 \nonumber \\
&  \geq &   0.  
  \end{eqnarray}
For the completeness, we also observe  
  \begin{equation}
\label{eq:dellammr5a0}
  \mE \cG (\cX^{(a_1,b_1)})\cG (\cX^{(a_1,b_2)})
 -
\mE \cG_l (\cX^{(a_1,b_1)})\cG_l (\cX^{(a_1,b_2)} )
  = 
 \lp c^2- \lp \x^{(a_1)} \rp^T S^TS \x^{(a_1)}  \rp \lp 1 - \lp \y^{(b_1)} \rp^T\y^{(b_2)}\rp
  =   0,
  \end{equation}
  where the last equality follows since $\|S\x^{(a_i)}\|_2=c^2$ for $i=1,2$. Additionally, we also have
    \begin{equation}
\label{eq:dellammr5a0a0}
  \mE \cG (\cX^{(a_1,b_1)})\cG (\cX^{(a_1,b_1)})
 -
\mE \cG_l (\cX^{(a_1,b_1)})\cG_l (\cX^{(a_1,b_1)} )
  = 
 \lp c^2- \lp \x^{(a_1)} \rp^T S^TS \x^{(a_1)}  \rp \lp 1 - \lp \y^{(b_1)} \rp^T\y^{(b_1)}\rp
  =   0,
  \end{equation}
where the last equality follows since $\|S\x^{(a_i)}\|_2=c^2$ and/or $\y^{(a_i)}\in\mS^m$ for $i=1,2$.
  
We recall on the following (complete) version of Theorem 1.1 from \cite{Gordon85}.  
\begin{theorem}(\cite{Gordon85})
\label{thm:minGordonpos1} Let $X_{ij}$ and $Y_{ij}$, $1\leq i\leq n$, $1\leq j\leq m$, be two centered Gaussian processes which satisfy the following inequalities for all choices of indices
\begin{enumerate}
\item $\mE(X_{ii}^2)=\mE(Y_{ii}^2)$
\item $\mE(X_{ij}X_{il})\geq \mE(Y_{ij}Y_{il})$
\item $\mE(X_{ij}X_{lk})\leq \mE(Y_{ij}Y_{lk})$, $i\neq l$.
\end{enumerate}
 Then
\begin{equation*}
\mE(\min_{i} \max_{j} X_{ij})\leq \mE(\min_i \max_{j} Y_{ij}) \quad  \Longleftrightarrow \quad \mE(\max_{i}\min_{j} X_{ij})\geq \mE(\max_i \min_{j} Y_{ij}).
\end{equation*}
\end{theorem}

\begin{remark}
\label{rem:rem4}
  As stated earlier, Theorem \ref{thm:Gordonpos1} is technically a special case of Theorem  \ref{thm:minGordonpos1}. However, it is known as Slepian lemma \cite{ Slep62} and it existed on its own long before introduction of Theorem  \ref{thm:minGordonpos1} in \cite{Gordon85}. For further developments and more general variants of which Theorem  \ref{thm:minGordonpos1} is a special case see, e.g., \cite{Stojnicgscompyx16,Stojnicgscomp16}.
\end{remark}

With correspondence $Y\leftrightarrow\cG$ and $X\leftrightarrow\cG_u$ one can apply  Theorem \ref{thm:Gordonpos1} to processes $\cG(\cdot)$ and  $\cG_l(\cdot)$. As a result, the following is obtained
\begin{align}\label{eq:dellammt5a1a0}
&  &\mE \max_{\cX^{(a_1,b_1)}} \cG_l(\cX)  & \geq \mE \max_{\cX^{(a_1,b_1)}} \cG(\cX)
\nonumber \\
\Longleftrightarrow & & \mE \min_{\x\in\mS^d,\|S\x\|_2=c}\max_{\y\in\mS^n} \lp \sum_{i=1}^n \sum_{j=1}^m A_{i,j}\s_i\x_i\y_j  + c g\rp & \geq \mE \min_{\x\in\mS^d,\|S\x\|_2=c}\max_{\y\in\mS^n} \lp c\lp\g^{(1)}\rp^T\y +   \lp\g^{(2)}\rp^TS\x \rp
\nonumber \\
\Longleftrightarrow & & \mE \min_{\x\in\mS^d,\|S\x\|_2=c}\max_{\y\in\mS^n}  \y^TAS\x & \geq \mE \min_{\x\in\mS^d,\|S\x\|_2=c}  \lp\| c\g^{(1)}\|_2 +   \lp\g^{(2)}\rp^TS\x \rp.
 \end{align}
Connecting further (\ref{eq:dellam12}) and (\ref{eq:dellammt5a1a0}), we then also find
\begin{eqnarray}\label{eq:dellammt5a1a1}
 \mE  \xi^-(c)  & \geq &  \frac{c}{\sqrt{n}} \mE\| \g^{(1)}\|_2  +  \frac{1}{\sqrt{n}}   \mE \min_{\x\in\mS^d,\|S\x\|_2=c }     \lp\g^{(2)}\rp^TS\x
 \nonumber \\
 & \geq &  \frac{c}{\sqrt{n}} \sqrt{\mE\| \g^{(1)}\|_2^2}  +  \frac{1}{\sqrt{n}}  \mE \min_{\x\in\mS^d,\|S\x\|_2=c }     \lp\g^{(2)}\rp^TS\x
  \nonumber \\
 & \geq &  c -  \frac{1}{\sqrt{n}}  \mE \max_{\x\in\mS^d,\|S\x\|_2=c}     \lp\g^{(2)}\rp^TS\x,
 \end{eqnarray}
which, together with (\ref{eq:thm1eq1}), gives  (\ref{eq:dellamthm1eq2}) and completes the proof.
\end{proof}

Finally, we are now in position to complete the story and establish spectral radius analogue to Theorems \ref{thm:thm3} and  \ref{thm:lbstrthm2}.

\begin{theorem}
\label{thm:dellamthm3}
Assume the setup of Theorem \ref{thm:thm3} with $\phi_1(\cdot)$, $\phi_2(\cdot)$, and $\phi_3(\cdot)$ from (\ref{eq:hrd17a1}) and (\ref{eq:hrd19}). Let $\hat{\gamma}_x$ satisfy (\ref{eq:thm3eq1}). One then has  
\begin{equation}\label{eq:dellamthm3eq2}
\delta(\alpha) =
 \lim_{d\rightarrow \infty}\mE \| \hat{\Sigma} -\Sigma \|_2 
=
 \lim_{d\rightarrow \infty}\mE \lambda_n\lp \hat{\Sigma} -\Sigma   \rp 
  = \hat{\delta}_u(\alpha) 
  =  \lim_{d\rightarrow \infty} \frac{\hat{\gamma}_x\sqrt{\phi_1(\hat{\gamma}_x)}} {\sqrt{\phi_1(\hat{\gamma}_x)} -\sqrt{\alpha} }.
       \end{equation}
 \end{theorem}

\begin{proof} 
A combination of (\ref{eq:dellam10}) and (\ref{eq:dellam13}) gives
\begin{eqnarray}\label{eq:dellamxpr1}
\left |\lambda_1 \lp\hat{\Sigma} - \Sigma\rp \right | = \lambda_n \lp\Sigma -\hat{\Sigma}\rp
     =  
 \max_{c} \lp  c^2 - \lp \xi^-(c)\rp^2 \rp.
\end{eqnarray}
From  (\ref{eq:dellamthm1eq2}) we then find
\begin{eqnarray}\label{eq:dellamxpr2}
\lim_{d\rightarrow \infty} \mE \left |\lambda_1 \lp\hat{\Sigma} - \Sigma\rp \right |   
     =  
 \max_{c} \lp  c^2 - \lp  \max \left \{c -  \frac{1}{\sqrt{\alpha}}\lim_{d\rightarrow \infty}\frac{1}{\sqrt{d}} \mE L(c),0 \right \} \rp^2 \rp,
\end{eqnarray}
and
\begin{equation}\label{eq:dellamxpr3}
\lim_{d\rightarrow \infty} \mE \left |\lambda_1 \lp\hat{\Sigma} - \Sigma\rp \right |   
     =   \max_c \hspace{.03in} \omega(c), 
     \end{equation}
where
\begin{equation}\label{eq:dellamxpr4}
\omega(c)     =  \begin{cases}
           \frac{2c}{\sqrt{\alpha}} \lim_{d\rightarrow \infty}\frac{1}{\sqrt{d}} \mE L(c)
           -
           \frac{1}{\alpha} \lp\lim_{d\rightarrow \infty}\frac{1}{\sqrt{d}} \mE L(c)\rp^2, & \mbox{if } c \geq   \frac{1}{\sqrt{\alpha}}\lim_{d\rightarrow \infty}\frac{1}{\sqrt{d}} \mE L(c) \\
          c^2, & \mbox{otherwise}.
        \end{cases}
 \end{equation}
Combining  (\ref{eq:dellamxpr3}) and  (\ref{eq:dellamxpr4})  with (\ref{eq:dellam5}), one arrives at
\begin{eqnarray}\label{eq:dellamxpr5}
\lim_{d\rightarrow \infty} \mE \left |\lambda_1 \lp\hat{\Sigma} - \Sigma\rp \right |
\leq \lim_{d\rightarrow \infty } \mE  \lambda_n\lp \hat{\Sigma} -\Sigma   \rp .
 \end{eqnarray}
The proof is then completed after one recognizes that (\ref{eq:dellamthm3eq2}) follows automatically from 
(\ref{eq:amat1a0a5}), (\ref{eq:amat1a0a6}), (\ref{eq:thm3eq2}),  (\ref{eq:thm3eq3}), (\ref{eq:dellamxpr5}), and main results of Theorems \ref{thm:lbstrthm2}  and \ref{thm:lbstrthm3}.  
 \end{proof}

\subsection{Practical numerical utilization}
\label{sec:numeval}

The results from the previous sections provide  characterization of
$\lambda_n\lp \hat{\Sigma} -\Sigma   \rp $ and $\delta(\alpha)$. In Figure \ref{fig:fig2} we compare these theoretical predictions with simulated results. As in Figure \ref{fig:fig1}, for different dimensions $d$, we selected $\s=\mbox{linspace}_d[0.5,1]$ as a vector of $d$ equally spaced components from $[0.5,1]$ interval. While the theoretical framework requires $d\rightarrow \infty$,  Figure \ref{fig:fig2}  shows that an excellent agreement between the theory and simulations is achieved already for $d$ on the order of thousand.

  \begin{figure}[h]
\centering
\centerline{\includegraphics[width=.87\linewidth]{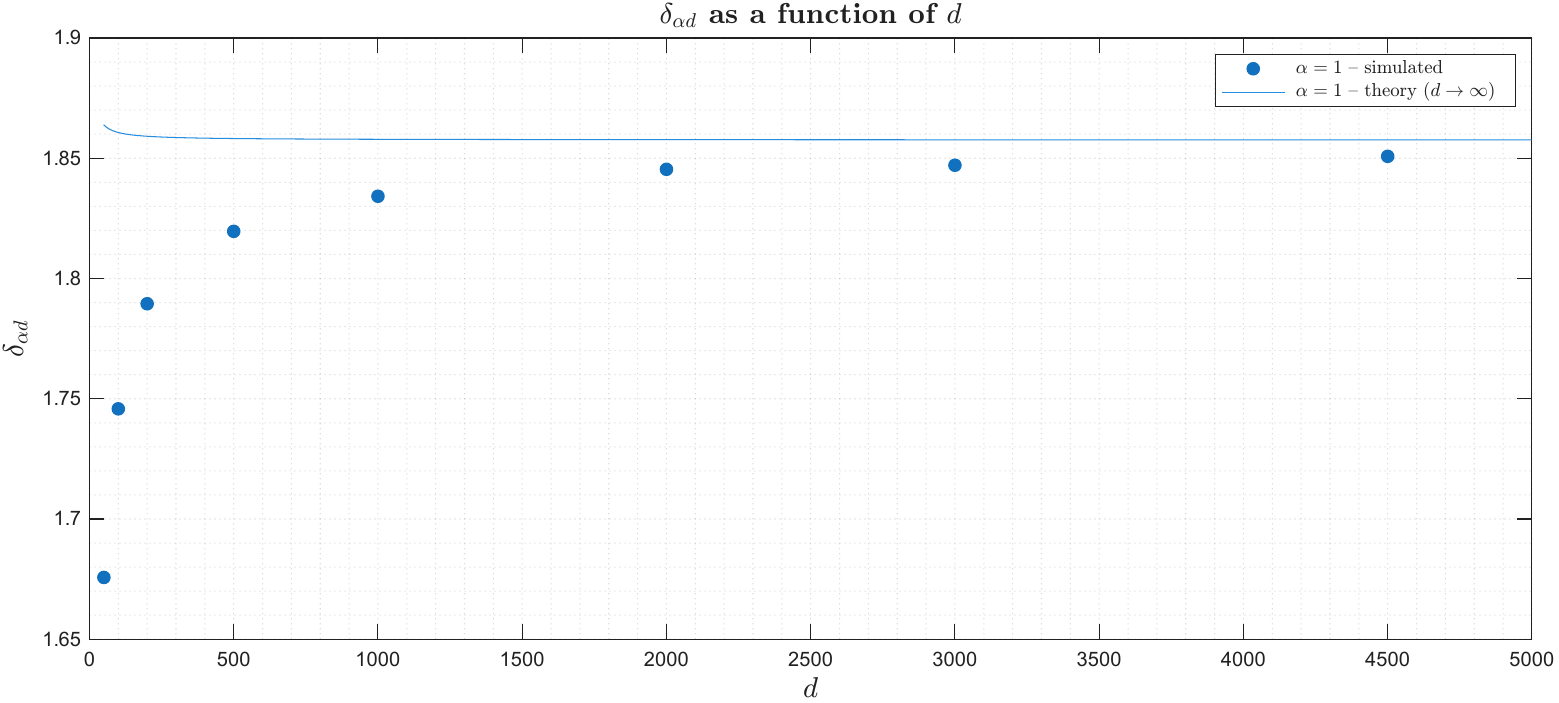}}
\caption{ Sample covariance error, $\delta_n=\delta_{\alpha d} =\mE \|\hat{\Sigma} -\Sigma\|_2 $, as a function of $d$; $\alpha=1$, i.e., $n=d$; $\s=\mbox{linspace}_d[0.5,1]$ }
\label{fig:fig2}
\end{figure}

The conducted analysis allows to obtain very precise estimation error characterizations. Consequently, it enables to accurately predict concrete effects of the increased number of samples. We show this in 
Figure \ref{fig:fig3} where $\alpha$ is successively increased. As can be seen, the simulated results again closely follow the theoretical predictions, although the utilized dimensions are relatively small   compared to the infinitely-dimensional setup required by the theoretical analysis.

  \begin{figure}[h]
\centering
\centerline{\includegraphics[width=.87\linewidth]{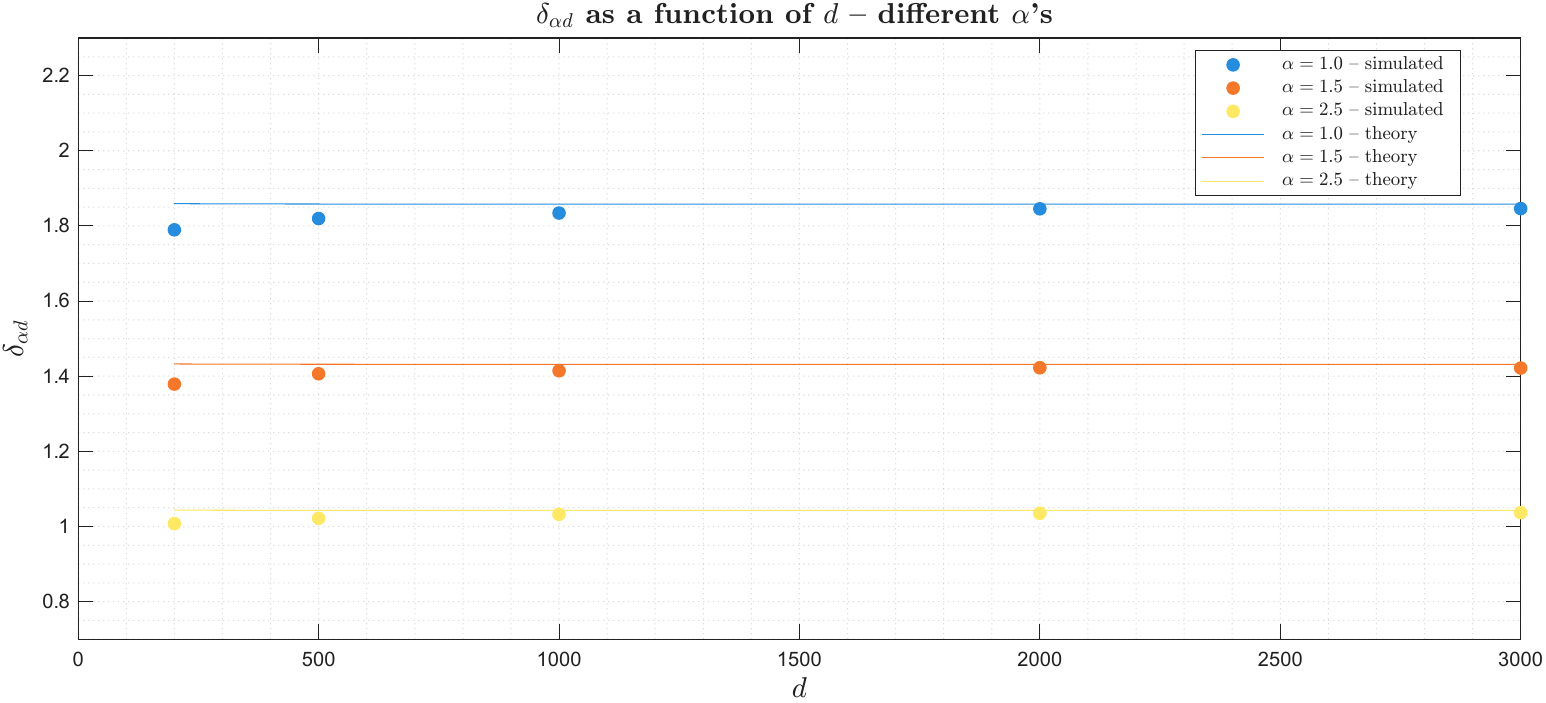}}
\caption{ Sample covariance error, $\delta_n=\delta_{\alpha d} =\mE \|\hat{\Sigma} -\Sigma\|_2 $, as a function of $d$; varying sample complexity $n$, i.e., varying  $\alpha=\frac{n}{d}$; $\s=\mbox{linspace}_d[0.5,1]$ }
\label{fig:fig3}
\end{figure}

\section{Conclusion}
\label{sec:conc}

We studied the sample covariance error of centered Gaussians. To
move beyond scaling characterizations and determine the precise limiting value of the error's spectral norm, we have developed a generic framework based on Random Duality Theory (RDT). The framework consists of three key components: (1)  Deriving closed-form, explicit RDT-based upper bounds; (2)  Introducing a novel bilinear-quadratic RDT lower-bounding mechanism; and (3)  Combining this lower-bounding mechanism with a 2-replica systems bounding strategy to demonstrate that the upper bounds can be matched in large-dimensional contexts. Our theoretical predictions show excellent agreement with the results obtained through numerical evaluations and simulations.

The developed framework is highly generic and offers significant potential for further extensions, including studying nearly all associated problem variants considered in recent literature and beyond. While the conceptual foundations remain the same, the specific underlying technical considerations are problem-dependent and will be discussed in future work.

%
%
%
%
%
%
%

\begin{singlespace}
\bibliographystyle{plain}
\bibliography{nflgscompyxRefs}
\end{singlespace}

\end{document}